\documentclass{amsart}

\usepackage{geometry}
\usepackage{amsmath}
\usepackage{amssymb}
\usepackage{amsthm}
\usepackage{mathtools}
\usepackage{tikz,ifthen}

\usepackage{float}

\theoremstyle{definition}
\newtheorem{thm}{Theorem}[section]
\newtheorem{cor}[thm]{Corollary}
\newtheorem{prop}[thm]{Proposition}
\newtheorem{lem}[thm]{Lemma}

\newtheorem{defn}[thm]{Definition}

\newtheorem{exmp}[thm]{Example}

\newtheorem{rmk}[thm]{Remark}

\newtheorem{obs}[thm]{Observation}

\newtheorem*{ack*}{Acknowledgements}


\begin{document}

\title{Geometric Symmetric Chain Decompositions}

\author{Stefan David}
\address{University of Cambridge, Trinity College, Trinity Street, CB21TQ}
\email{sd637@cam.ac.uk}
\thanks{The first author would like to thank Trinity College, Cambridge, for providing travel funding.}

\author{Hunter Spink}
\address{Harvard University, Cambridge, 1 Oxford Street, 02138}
\email{hspink@math.harvard.edu}
\thanks{The second author would like to thank Harvard University for providing travel funding.}

\author{Marius Tiba}
\address{University of Cambridge, Trinity Hall, Trinity Lane, CB21TJ}
\email{mt576@cam.ac.uk}
\thanks{The third author would like to thank Trinity Hall, Cambridge, for support through the Trinity Hall Research Studentship.}

\begin{abstract}
We create a framework for studying symmetric chain decompositions of families of finite posets based on the geometry of polytopes. Our framework unifies almost all known results regarding symmetric chain decompositions of the Young posets $L(m,n)$ --- arising as cells in the Bruhat decomposition of quotients of $SL_{m+n+1}$ --- and yields unexpected new results. The methods we provide are geometric in nature, systematic, and totally amenable to human analysis. This allows us to discover new phenomena which are impenetrable to casework and brute force computer search. In particular, our method yields perfect and near perfect decompositions of various families of posets, which are intractable by known methods. A fundamental tool we use is geometrical projection, which in our framework cleanly unifies many different types of induction; as we move a point from which we project between faces of our polytope, we alter the type of induction. Moreover, projection allows us to decrease dimension and therefore obtain a clear geometric intuition. We also provide additional tools for producing decompositions, and discuss how the various decompositions behave under products.
\end{abstract}

\maketitle

\section{Introduction}
In 1970, in his seminal paper \cite{Kleitman}, Kleitman gave perhaps the first non-trivial application of symmetric chain decompositions. There, he used a recursive construction of symmetric chain decompositions of the hypercube $\{0,1\}^n$ to solve the Littlewood-Offord problem on concentrations of sums of Bernoulli random variables in arbitrary dimensions. Since then, a large number of interesting rank-symmetric posets have been shown to admit symmetric chain decompositions and a plethora of necessary conditions and sufficient conditions for such decompositions to exist have appeared in the literature (see e.g. Griggs \cite{Griggs} and Stanley \cite{Stanley2}).

Recently, focus in the area has shifted to using methods of Algebraic Geometry to construct new such posets and decompositions (see e.g. Dhand \cite{dhand2}). The first results in this direction started in the groundbreaking paper of Stanley \cite{stanley}. Posets arising from these constructions typically have the form $\{0,1\}^n/G$, for $G$ a group acting on the hypercube. In this direction, the only known results treat special group actions such as the natural action of $\mathbb{Z}/n\mathbb{Z}$ as in Dhand, Hersh-Schilling and Jordan \cite{Dhand, Hersh, Jordan}, or generalizations as in Duffus \cite{Duffus}. However, the posets considered in Stanley's original paper, which have a more geometrical interpretation, remain largely open as to the existence of symmetric chain decompositions. These are the Young posets $L(m,n)$, which are a central object of study in the present paper (called such because they arise from the poset of Young diagrams under the natural inclusion relation). Interest in symmetric chain decompositions of the Young posets arose from Stanley's reinterpretation of the posets as labelings of the cells in certain cellular decompositions (induced by the Bruhat decomposition) of quotients of $SL_{m+n+1}$ by certain maximal parabolic subgroups \cite{stanley}. By applying the Lefschetz hyperplane theorem, a complexified analogue of a symmetric chain decomposition was shown to exist for all $m$ and $n$, bolstering evidence that these posets admit symmetric chain decompositions.

Perhaps the only general result on symmetric chain decompositions of Young posets is due to O'Hara \cite{Ohara}, giving an explicit combinatorial proof of unimodality. Lindstrom, West, and Wen \cite{ Lindstrom,west,xiangdong} give various decompositions of $L(m, n)$ for $m \le 4$; however, decomposing $L(m,n)$ for $m \ge 5$ remains a wide open problem.

The natural geometric interpretation of the posets $L(m,n)$ has never been properly exploited as previous ad-hoc algebraic methods of construction and verification are difficult without computer assistance. This paper provides a bridge between the combinatorics of symmetric chain decompositions and the geometry of polytopes.

A good place to start would be to describe our natural geometric interpretation of $L(m,n)$. Classically, we may view it as the set of $m$-tuples of integers $0 \le x_1 \le \ldots \le x_m \le n$, under the partial ordering $(x_i) \le (y_i)$ if $x_i \le y_i$ for all $i$. Note that if we define the polytope $$L(m)=\{(x_1,\ldots,x_m) \in \mathbb{R}^m \mid 0 \le x_1 \le \ldots \le x_m \le 1\},$$ endowed with the partial ordering $(x_i) \le (y_i)$ if $x_i \le y_i$ for all $i$, then $$L(m,n)=L(m) \cap \frac{1}{n} \mathbb{Z}^m.$$

This already yields the observation that for any $k \in \mathbb{N}$, $L(m,n)$ naturally lies inside $L(m,kn)$ by treating them as subposets of $L(m)$, in a way that preserves the notion of ``symmetry about the middle rank'' given by their respective rank functions. Consequently, we could consider the compatibility of their symmetric chain decompositions. Before pursuing this idea, let's frame our discussion in a more general context.

Given a polytope $P$, we define $P(n)$ to be the poset whose elements are the points of denominator $n$ contained inside $P$, with the partial order $(x_i) \le (y_i)$ if $x_i \le y_i$ for all $i$. In the above discussion, we had $P=L(m)$, and so $L(m)(n)=L(m,n)$. The most natural compatibility condition between symmetric chain decompositions of $P(n)$ and $P(kn)$ one might impose is that the symmetric chain decomposition of $P(n)$ is the restriction of the symmetric chain decomposition of $P(kn)$. It turns out that from a geometric perspective, this condition is quite natural, as we shall see below.

As an example, take $P$ to be either a $4\times 6$ rectangle or a $4\times 4$ right triangle, axis-aligned with the bottom left point at the origin and consider the poset $P(1)$ as depicted in Figure 1. We indicate in thick black lines one possible symmetric chain decomposition of each of these posets.
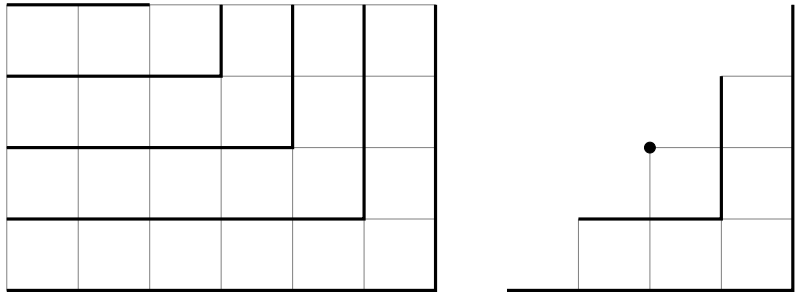
\begin{figure}[H]
\centering
\begin{tikzpicture}[scale=0.95]
\draw[step=1cm,gray,very thin] (0,0) grid (6,4);
\draw[very thick] (0,0)--(6,0)--(6,4);
\draw[very thick] (0,1)--(5,1)--(5,4);
\draw[very thick] (0,2)--(4,2)--(4,4);
\draw[very thick] (0,3)--(3,3)--(3,4);
\draw[very thick] (0,4)--(2,4);
\draw[gray,very thin] (7,0)--(11,0);
\draw[gray,very thin] (8,1)--(11,1);
\draw[gray,very thin] (9,2)--(11,2);
\draw[gray,very thin] (10,3)--(11,3);
\draw[gray,very thin] (11,0)--(11,4);
\draw[gray,very thin] (10,0)--(10,3);
\draw[gray,very thin] (9,0)--(9,2);
\draw[gray,very thin] (8,0)--(8,1);
\draw[very thick] (7,0)--(11,0)--(11,4);
\draw[very thick] (8,1)--(10,1)--(10,3);
\draw[black,fill=black] (9,2) circle (0.5ex);
\end{tikzpicture}
\caption{Symmetric chain decompositions of $P(1)$ for a $4 \times 6$ rectangle and $4 \times 4$ right triangle.}
\end{figure}

Suppose we want to decompose $P(4)$ next. Using Figure 1 as inspiration, we can produce the decompositions depicted in Figure 2.

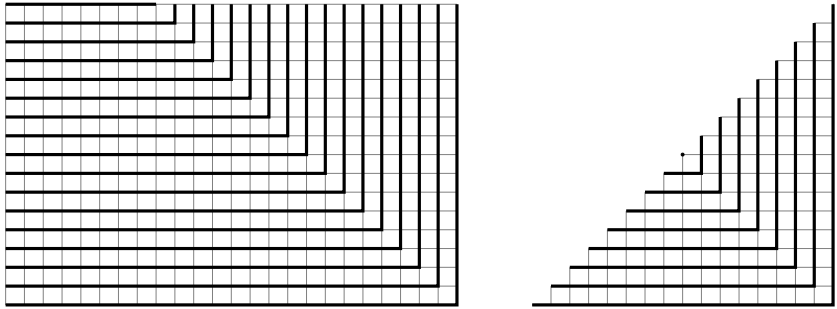
\begin{figure}[H]
\centering
\begin{tikzpicture}
\draw[step=0.25cm,gray,very thin] (0,0) grid (6,4);
\foreach \x in {0,1,...,16}
{
\draw[very thick] (0,\x/4)--(6-\x/4,\x/4)--(6-\x/4,4);
}

\foreach \x in {0,1,...,16}
{
\draw[gray,very thin] (7+\x/4,\x/4)--(11,\x/4);
\draw[gray,very thin] (11-\x/4,0)--(11-\x/4,4-\x/4);
}
\foreach \x in {0,1,...,8}
{
\draw[very thick] (7+\x/4,\x/4)--(11-\x/4,\x/4)--(11-\x/4,4-\x/4);
}
\draw[black,fill=black] (9,2) circle (0.125ex);

\end{tikzpicture}
\caption{Symmetric chain decompositions of $P(4)$ for a $4 \times 6$ rectangle and $4 \times 4$ right triangle.}
\end{figure}

Note that our inspiration has naturally led to the above compatibility condition being satisfied between $P(1)$ and $P(4)$. It is tempting to try and encode the pattern which allows us to produce symmetric chain decompositions of the various posets $P(n)$ by sight, as depicted in Figure 3.
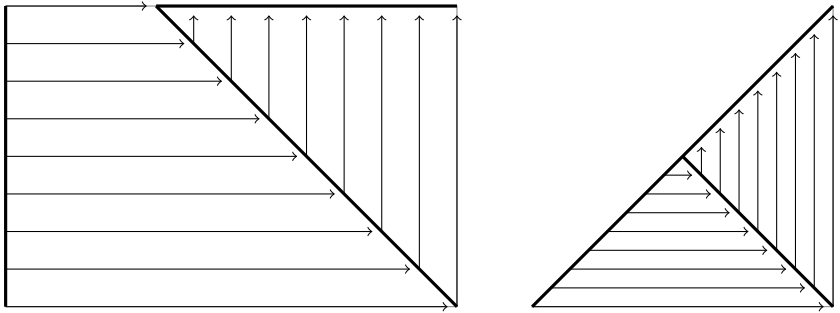
\begin{figure}[H]
\centering
\begin{tikzpicture}
\draw[gray, very thin] (0,0)--(6,0)--(6,4)--(0,4)--cycle;
\draw[very thick] (2,4)--(6,0);
\draw[very thick] (0,0)--(0,4);
\draw[very thick] (2,4)--(6,4);
\foreach \x in {0,1,2,3,4,5,6,7,8}
{
\draw [->, shorten >=0.125cm] (0,\x/2)--(6-\x/2,\x/2);
}
\foreach \x in {1,2,3,4,5,6,7,8}
{
\draw [->, shorten >=0.125cm] (2+\x/2,4-\x/2)--(2+\x/2,4);
}
\draw[gray, very thin] (7,0)--(11,4)--(11,0)--cycle;
\draw[very thick] (7,0)--(11,4);
\draw[very thick] (9,2)--(11,0);
\foreach \x in {0,1,2,3,4,5,6,7}
{
\draw [->, shorten >=0.125cm] (7+\x/4,\x/4)--(11-\x/4,\x/4);
\draw [->, shorten >=0.125cm] (11-\x/4,\x/4)--(11-\x/4,4-\x/4);
}
\end{tikzpicture}
\caption{Geometric symmetric chain decompositions for a $4 \times 6$ rectangle and $4 \times 4$ right triangle.}
\end{figure}

It appears that in the ``limit'', we have created some sort of analogue of a symmetric chain decomposition for the polytope $P$. Whatever this symmetric chain decomposition of $P$ is, if it correctly induces symmetric chain decompositions on the various $P(n)$, then compatibility should be a natural consequence.

We shall now be more precise, and define in full generality a symmetric chain decomposition of a (not necessarily full-dimensional) non self-intersecting polytope $P$ in $\mathbb{R}^m$. To do this, we have to define what a chain is, and what it means for a chain to be symmetric. Recall the points in $P$ are partially ordered by setting $(x_i) \le (y_i)$ if $x_i \le y_i$ for all $1 \leq i \leq m$. Say that the \textit{rank} of a point is the sum of its coordinates, and as a matter of convenience, say the \textit{rank} of a polytope $P$ is the sum of the maximum and minimum ranks of points in $P$. Hence we regard a chain that skips no ranks as a curve inside $P$ on which the rank function is increasing. For our purposes we demand that a chain be made up of broken line segments with sides parallel to the coordinate axes (as in Figure $3$), and that it contains either both or neither of its endpoints (a natural condition in the context of symmetric chains).

To that end, we define a \textit{closed chain} inside a polytope $P$ to be given by the union of closed line segments between $v_0$ and $v_1$, $v_1$ and $v_2$, $\ldots$, $v_{k-1}$ and $v_k$, for $v_i$ some points inside $P$, with $v_i-v_{i-1}$ a positive multiple of some standard basis vector. A single point of $P$ is considered to be a closed chain by taking $k=0$. We define an \textit{open chain} similarly, but it omits the two endpoints $v_0$ and $v_k$. A chain (open or closed) is called \textit{symmetric} if the sum of the ranks of $v_0$ and $v_k$ is equal to the rank of $P$. We say a \textit{symmetric chain decomposition} is a partition of $P$ into symmetric chains.

In this paper, we consider the interplay between symmetric chain decompositions of polytopes $P$, and symmetric chain decompositions of the underlying posets $P(n)$.

The diagram in Figure $3$ depicts a much more ``geometric'' decomposition than what one might expect a random symmetric chain decomposition of $P$ to look like. Indeed, the chains appear to move cohesively in groups rather than separately --- this is precisely what we encode in our definition of a ``geometric symmetric chain decomposition'' (presented in Section~2).

Roughly speaking, a ``geometric symmetric chain decomposition'' consists of a collection of simplices called ``starting sets'', which are continuously sheared in a sequence of positive coordinate directions, forming a sequence of ``swipes'', and changing directions at certain intermediate ``turning sets''. This is done in such a way that the resulting chains that are swept out form a symmetric chain decomposition of the polytope $P$. In both of the decompositions from Figure 3, we have three turning sets which are bolded closed line segments, one of which is the starting set, and two swipes.

What will we gain by codifying this notion? Ideally, we will be able to take the geometric picture (produced by observing small examples as above, or purely through geometric insight), and read off instructions on how to decompose $P(n)$ into symmetric chains for any given $n$ by restricting the chains in $P$ to $P(n)$.

We return to the polytope $L(m)$ after a small digression, asking which conditions need to be imposed on the ``geometric symmetric chain decomposition'' of a polytope $P$ to hope for a decomposition of the underlying posets. Surprisingly, the conditions are rather mild, and address three problems which may arise. Firstly, as we shall see below, a certain density condition must be satisfied, so that there are asymptotically the same number of discrete chains coming towards a turning set as there are leaving it. Secondly, we demand that the discrete chains actually land on the turning sets without hopping over, otherwise issues arise for chains near the boundary. Thirdly, it might be that a turning set only contains points of $P(n)$ for $n$ a multiple of a certain number. For the first two of these problems, we call the corresponding conditions on the turning sets the ``weak'' and ``strong'' hyperplane conditions respectively. For the third problem we introduce the notion of ``complexity'' of a turning set. As it turns out, the second problem is contained in the first problem, so the ``strong'' hyperplane condition implies the ``weak'' hyperplane condition.

In Figure~4 on the left, we can see the second problem occuring: in $P(n)$ the points on the bottom line never biject to the points on the slanted line. Hence if such a figure appeared in the decomposition, it would not decompose any discrete poset perfectly. Even if we don't require the chains to exactly land on the slanted line (being cavalier about boundary concerns), we see the first problem occurring:  for $n$ very large, the number of chains leaving the slanted line is approximately $\frac{2}{3}$ times the number of chains entering the slanted line, so there is no way to match them up! In Figure~4 on the right we see an example of the third problem: in $P(n)$ the points on the bottom line biject to the points on the top line only for $n$ a multiple of $2$.

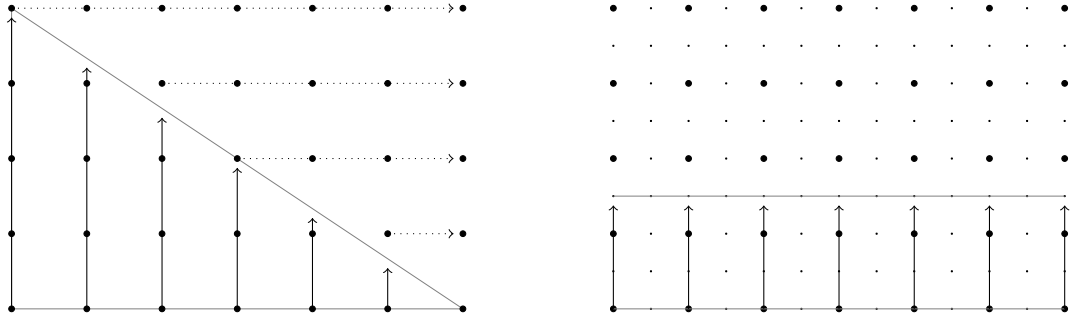
\begin{figure}[H]
\centering
\begin{tikzpicture}
\draw[gray,very thin] (0,0)--(6,0)--(0,4)--cycle;
\foreach \x in {1,2,...,6}
{
\draw [->, shorten >=0.125cm] (6-\x,0)--(6-\x,2*\x/3);
}
\draw[->,shorten >=0.125cm, dotted]  (0,4)--(6,4);
\draw[->,shorten >=0.125cm, dotted]  (2,3)--(6,3);
\draw[->,shorten >=0.125cm, dotted]  (3,2)--(6,2);
\draw[->,shorten >=0.125cm, dotted]  (5,1)--(6,1);
\foreach \x in {0,1,...,6}
	\foreach \y in {0,1,...,4}
    {
		\draw[black,fill=black] (\x,\y) circle (0.25ex);
        \draw[black,fill=black] (\x+8,\y) circle (0.25ex);
    }
\foreach \x in {0,1,...,12}
\foreach \y in {0,1,...,8}
{
\draw[black,fill=gray] (\x/2+8,\y/2) circle (0.05ex);
}
\draw[gray,very thin] (8,0)--(14,0);
\draw[gray,very thin] (8,1.5)--(14,1.5);
\foreach \x in {0,1,2,...,6}
{
\draw [->, shorten >=0.125cm] (\x+8,0)--(\x+8,1.5);
}
\end{tikzpicture}
\caption{Examples of obstructions to discretization.}
\end{figure}

It is very surprising that there are no other problems that can possibly arise. By one of our general main results, Theorem~3.3, given a ``geometric symmetric chain decomposition'' with the ``strong hyperplane condition'', we get induced symmetric chain decompositions of all $P(n)$ for $n$ which are a multiple of the ``complexities'' of the turning sets. In particular, if the ``complexities'' of all turning sets are $1$, then we get induced symmetric chain decompositions of $P(n)$ for all $n$.

Suppose that we had instead only asked for an \textit{asymptotic} decomposition of $P(n)$, i.e. a symmetric chain decomposition of all but $O(\frac{1}{n})$ of the points in $P(n)$. Then intuitively, we would only require a ``geometric symmetric chain decomposition'' in which codimension $1$ concerns can be ignored --- we call this an ``asymptotic geometric symmetric chain decomposition'' and define it rigorously in Section 2. Additionally, concerns near the boundary can be ignored in this situation, so we only need to care about the chain density condition.  By another general main result, Theorem~3.6, provided $P$ is of the same dimension as the ambient space, then given an ``asymptotic geometric symmetric chain decomposition'' of $P$ satisfying the ``weak hyperplane condition'', we have an asymptotic decomposition of $P(n)$.

As it turns out, the ``weak hyperplane condition'' is a purely geometric condition, while the ``strong hyperplane condition'' is a number theoretic one. Consequently, it is most natural to seek decompositions with the ``weak hyperplane condition'' first, before checking if the ``strong hyperplane condition'' is satisfied.

Finally, suppose we have a ``geometric symmetric chain decomposition'' with rational ``turning sets''. Even though the densities between the different ``swipes'' differ, by another general main result, Theorem~3.5, we have that there exists a number $M$ such that for all $n$, there is a disjoint collection of symmetric chains in $P(nM)$ which covers $P(n) \subseteq P(nM)$. Intuitively, this happens because every time a symmetric chain in $P$ which passes through an element of $P(n)$ hits a turning set, the size of the denominator needed to define the point of intersection only increases by a bounded factor.

We now return to the Young posets $L(m,n)$ and the polytopes $L(m)$. In Section~4, we fully apply our framework to study the polytopes $L(m)$ and other related polytopes. A fundamental tool for our investigations is the notion of projection. Given a polytope $P$ and a point $x \in P$ around which $P$ is star-shaped, we define the \textit{projected polytope} $Q$ to be the result of projecting $P$ onto its boundary, or equivalently the union of all faces of $P$ which do not contain $x$. Given a ``geometric symmetric chain decomposition'' of $Q$, we can \textit{cone off at x}, i.e. reverse the procedure to get an induced decomposition of $P$. Of course, as $Q$ only depends on which face $x$ lies in, we can first decompose $Q$, and then solve for the coordinates of $x$ in order to satisfy the various hyperplane conditions.

Combinatorially, projecting from $x$ encodes mathematical induction: we are morally representing $P(n)$ as a union of homothetic shells about $x$, so that decomposing the shell associated to the boundary represents the inductive step. For a concrete example of this, suppose we have the following statement (as appeared for example in a similar form in West's paper \cite{west}): ``By embedding $L(4,n-2)$ inside of $L(4,n)$ as the set of points $1 \le x_1 \le \ldots \le x_4 \le n-1$, it suffices to decompose by induction the remaining set of points where $x_1=0$ or $x_4=n$''. Then geometrically, what we have done is projected from the point $(\frac{1}{2},\frac{1}{2},\frac{1}{2},\frac{1}{2})\in L(4)$, and reduced our considerations to the projected polytope, which is the union of the three-dimensional facets $x_1=0$ and $x_4=1$. More precisely, while projection occasionally encodes a straightforward inductive step as above, it is often the case that it implicitly handles a rather complicated modularity casework, arising exactly when $P(n)$ is not perfectly represented as the union of homothetic shells. This is a difficulty implicitly encountered in previous works, which in our framework is made obsolete.

The importance of projection is twofold. On the one hand, it cleanly unifies in our framework many different types of induction, some of which might even be too complicated to express cleanly from a purely combinatorial point of view. As the point we project from moves between faces of our polytope, we alter the type of mathematical induction (a phenomenon which to our knowledge has no purely combinatorial analogue). On the other hand, projection allows us to decrease dimension and therefore obtain a clear geometric intuition.

In Section 4, we present one ``geometric symmetric chain decomposition'' with the ``strong hyperplane condition'' for $L(3)$, and two for $L(4)$. Inspired by these decompositions, we give a ``geometric symmetric chain decomposition'' for $L(5)$ which happens to not satisfy the hyperplane conditions. From this, we immediately read off symmetric chain decompositions of $L(3,n)$ for all $n$, and two different symmetric chain decompositions of $L(4,n)$. For $L(5,n)$, we can conclude that there exists a collection of disjoint symmetric chains inside $L(5,27n)$ which covers $L(5,n)$.

One can also extract all of the symmetric chain decompositions of $L(3)$ and $L(4)$ that we discuss from the works of West \cite{west} and Wen \cite{xiangdong}. However, the methods used seem to be ad hoc, and computer brute force search in this vein has yielded no results for $m \ge 5$ \cite{xiangdong}. The methods of the present paper give rise to the construction of all of these decompositions in a systematic way by hand, giving a clear geometric insight into the decompositions of the polytopes and their underlying posets.

The similarities of all the decompositions we provide, including the ones which have appeared in the literature in some form, cannot be seen at the algebraic level, but become completely transparent with our method. It will be clear in Section 4.4 how the heuristic similarities of the decompositions enabled us to prove the new result above for $L(5,n)$.

As another application of our method, we further exploit these similarities by considering in Section 4.3 the polytope $L(4)_{a,b,c,d}$, a weighted version of $L(4)$ where the coordinate directions are stretched by factors $a,b,c,d$ respectively. In the discrete poset, rank symmetry fails asymptotically unless $a=d$ and $b=c$. However, if we remove a region from $L(4)_{a,b,c,d}$, we can potentially recover (asymptotic) rank symmetry. Using only basic Euclidean geometry, we are led to discover an ``asymptotic geometric symmetric chain decomposition'' with the ``weak hyperplane condition'' of an entire one-parameter family of polytopes formed by removing a naturally occurring region from $L(4)_{1,\frac{1}{t},t,1}$ to restore rank-symmetry. Even though these polytopes are $4$-dimensional, we are able to produce a clear geometrical intuition, aided largely by the dimension reducing property of projection.

Finally, we harken back to the beginning of the introduction, to discuss how Kleitman's solution to the Littlewood-Offord problem \cite{Kleitman} fits into the general context of products of symmetric chain decompositions, and how ``(asymptotic) geometric symmetric chain decompositions'' interact with products. Kleitman's ``duplication technique'' introduced in \cite{Kleitman} to inductively decompose $\{0,1\}^n=\{0,1\}^{n-1}\times \{0,1\}^1$ into symmetric chains is a special case of the general observation that to decompose the product of two posets with symmetric chain decompositions into symmetric chains, it suffices to decompose the product of each chain in the first decomposition with each chain in the second.

One would expect no difference between this situation and ours --- given symmetric chain decompositions of polytopes $P$ and $Q$, we would naturally expect to be able to take the products of chains in $P$ with chains in $Q$, and decompose the products into chains. However, a serious problem arises with the product of a closed chain and an open chain which prevents us from doing so. More surprisingly, we show in Example 6.10 that if we take $P$ to be the boundary of the unit square and $Q$ to be a unit line segment, then even though both $P$ and $Q$ have ``geometric symmetric chain decompositions'' with the ``strong hyperplane condition'', their product does not even have a symmetric chain decomposition! However, taking the product of an open chain with a closed chain is the only problem that arises. We show in Theorems 6.7 and 6.8 that the product of two ``asymptotic symmetric chain decompositions'' with the ``weak hyperplane condition'' has such a decomposition with the ``weak hyperplane condition'', and the product of two ``geometric symmetric chain decompositions'' with the ``strong hyperplane condition'' made up entirely of either all closed chains or all open chains has such a decomposition with the ``strong hyperplane condition''.

We also provide in Section 6 various easily calculable statistics of polytopes that might aid in producing other geometric decompositions with the ``weak'' and ``strong'' hyperplane conditions.

The structure of this paper is as follows.
\begin{itemize}
\item In Section 2, we introduce our definitions.
\item In Section 3, we state our main results.
\item In Section 4, we restrict ourselves to $L(m)$ and related polytopes, and detail the finding and verification of the decompositions up to $m=5$ (proving Theorems 3.1 and 3.7).
\item In Section 5, we prove the rest of our main results.
\item In Section 6, we fully investigate products of polytopes with geometric decompositions and develop some useful tools and techniques for finding various types of decompositions.
\item In Section 7, we make concluding remarks and give directions for further research.
\end{itemize}

\section{Definitions}

Fix a (not necessarily full-dimensional) polytope $P$. Given a simplex, a sub-simplex is the convex hull of some subset of vertices, and an open sub-simplex is the relative interior of a sub-simplex. We define a \textit{partial simplex} to be a simplex with some open sub-simplices removed that still contains the relative interior of the original simplex.

\begin{defn}
We define a \textit{swipe} to be a (convex) set which is formed by first specifying:
\begin{enumerate}
\item a coordinate direction $e_i$, and
\item two partial simplices $S^s$ and $S^e$ called the starting and ending \textit{turning sets} respectively
\end{enumerate}
such that the points of $S^s$ biject via translation by non-negative multiples of $e_i$ to the points of $S^e$.

We then take the swipe to be the union of all line segments in the $e_i$ direction connecting $S^s$ and $S^e$. The closure of a swipe is naturally a convex polytope.
\end{defn}

\begin{defn}
We say a sequence of swipes $S_0,\ldots, S_k$ is \textit{compatible} if for all $0\le i<k$, $S_i^e=S_{i+1}^s$.
\end{defn}

\begin{defn}
We define a \textit{real snake} $R$ to be a polytope formed by the union of a sequence of compatible swipes $S_0,\ldots, S_k$, which satisfies the non self-intersecting condition that if $i<j$ and $x \in S_i \cap S_j$, then $x$ lies in $S_i^e$, $S_j^s$, and $S_k^s,S_k^e$ for all $i<k<j$.

A point $x$ in $S_i$ is said to be \textit{i-halted} if either:
\begin{enumerate}
\item it lies in $S_k^s,S_k^e$ for all $k<i$ and also in $S_i^s$ or
\item it lies in $S_k^s,S_k^e$ for all $k>i$ and also in $S_i^e$.
\end{enumerate}

We define a \textit{fake snake} $F$ to be similarly formed by starting with a sequence of compatible swipes $S_0,\ldots, S_k$, and taking the union of $T_i=S_i\setminus\{\text{all }i\text{-halted points}\}$, provided the $T_i$ satisfy the non self-intersecting condition that if $i<j$ and $x \in T_i \cap T_j$, then $x$ lies in $T_i\cap S_i^e$, $T_j \cap S_j^s$, and $T_k\cap S_k^s, T_k \cap S_k^e$ for all $i<k<j$.
\end{defn}

\begin{rmk}
In the context of symmetric chain decompositions, real snakes will correspond to groups of closed chains, and fake snakes will correspond to groups of open chains.
\end{rmk}

\begin{defn}
In a compatible sequence $S_0,\ldots, S_k$, the points of $S_0^s$ canonically biject with the points of $S_k^e$. We say a snake contained in $P$ is $\textit{symmetric}$ if for every corresponding pair of vertices in $\bar{S_0^s}$ and $\bar{S_k^e}$, the sum of the ranks is equal to the rank of $P$.
\end{defn}

\begin{defn}
We define a $\textit{geometric symmetric chain decomposition}$ of $P$ to be a partition of $P$ into symmetric snakes.
\end{defn}

\begin{defn}
A geometric symmetric chain decomposition of the (not necessarily full-dimensional) polytope $P$ is said to satisfy the \textit{strong hyperplane condition} if given any turning set of a swipe of direction $e_i$, then some hyperplane containing this turning set can be written in the form $\sum a_jx_j = b$ with all $a_j$ integers, $b$ rational, and $a_i=\pm 1$. Call the denominator of $b$ the \textit{complexity} of the hyperplane. We define the \textit{complexity} of a turning set to be the minimum complexity over all such hyperplanes.
\end{defn}

\begin{defn}
We define an \textit{asymptotic geometric symmetric chain decomposition} of a polytope $P\subset \mathbb{R}^m$ of dimension $m$ to be a disjoint collection of snakes inside $P$ of dimension $m$ such that the set of points in $P$ which do not lie in any snake has codimension $\ge 1$.
\end{defn}

\begin{defn}
An asymptotic geometric symmetric chain decomposition of a polytope $P \subset \mathbb{R}^m$ of dimension $m$ is said to satisfy the \textit{weak hyperplane condition} if for any turning set inside a snake associated to two swipes of directions $e_i$ and $e_j$, the hyperplane that contains it can be written in the form $\sum a_kx_k = b$, with $a_i=a_j$.
\end{defn}

\begin{rmk}
The weak hyperplane condition precisely says that the density of discrete chains coming into a given turning set is the same as the density of discrete chains coming out of the turning set. The strong hyperplane condition additionally ensures that the discrete chains actually land on the turning set before turning.
\end{rmk}

Paramount to the decompositions we shall produce is the notion of projection.
\begin{defn}
Given a point $v$ of middle rank in $P$, such that $P$ is star-shaped around $v$, we define the \textit{projected polytope} $Q$ to be the union of all faces of $P$ in which $v$ doesn't lie. We say that $Q$ is obtained by \textit{projecting} $P$ from $v$, or that $P$ is \textit{coned off} from $Q$ by $v$.
\end{defn}

\begin{obs}
Any geometric symmetric chain decomposition, or symmetric chain decomposition of the projected polytope $Q$, naturally translates to such a decomposition of $P$.
\end{obs}

If the strong hyperplane condition is satisfied in the projected polytope $Q$, then it is not necessarily satisfied in $P$, which most notably happens for our decomposition of $L(5)$. However, it is necessary for the strong hyperplane condition to be satisfied in $Q$ for it to be satisfied in $P$.

\begin{rmk}If we sought to decompose $L(m,n)$ by induction, then we might consider the natural inclusion of $L(m,n-2)$ as $\{\frac{1}{n} \le x_1 \le x_2 \le \ldots \le x_m \le 1-\frac{1}{n}\}$ inside $L(m,n)$, decompose this set by induction, and then proceed to decompose the remainder (as was done be West for $m=4$ in \cite{west}). But this remainder simply corresponds to the poset associated to the polytope which we get by projecting $L(m)$ from $(\frac{1}{2},\frac{1}{2},\ldots,\frac{1}{2})$.
\end{rmk}

Projection will be used for all polytopes we decompose.

\begin{exmp}

We illustrate our methods for the polytope $L(2)$. First, we create a geometric symmetric chain decomposition, and then verify that the strong hyperplane condition is satisfied. For this example, set $\textbf{0}=(0,0)$, $\textbf{1}=(0,1)$, $\textbf{2}=(1,1)$, and write $\overline{\textbf{ij}}$ for the midpoint of $\textbf{i}\textbf{j}$.

We start by projecting from $\overline{\textbf{02}}$. The projected polytope of $L(2)$ is the union of the facets \textbf{01} and \textbf{12}. This has a geometric symmetric chain decomposition with turning sets $\textbf{0},\textbf{1},\textbf{2}$. The swipes \textbf{01} and \textbf{12} form a single real snake, as \textbf{0} and \textbf{2} have complementary ranks.

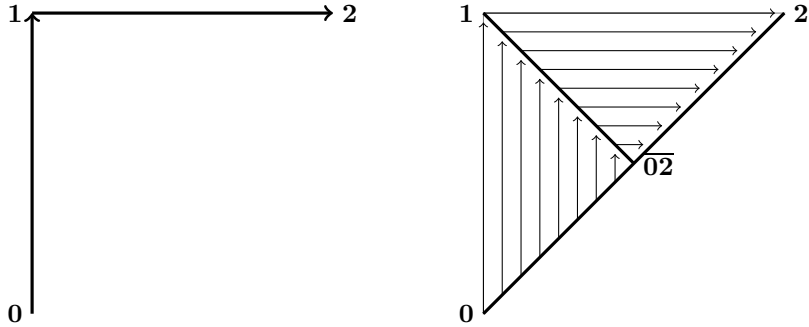
\begin{figure}[H]
\centering
\begin{tikzpicture}
\draw[gray, very thin] (0,0)--(4,4)--(0,4)--cycle;
\draw (0,0) node[left] {\textbf{0}};
\draw (0,4) node[left] {\textbf{1}};
\draw (4,4) node[right] {\textbf{2}};
\draw (2,2) node[right] {$\overline{\textbf{02}}$};
\draw[very thick] (0,0)--(4,4);
\draw[very thick] (2,2)--(0,4);
\foreach \x in {0,1,2,3,4,5,6,7}
{
\draw [->, shorten >=0.125cm] (\x/4,\x/4)--(\x/4,4-\x/4);
\draw [->, shorten >=0.125cm] (\x/4,4-\x/4)--(4-\x/4,4-\x/4);
}
\draw[very thick, ->] (-6,0)--(-6,4);
\draw[very thick, ->] (-6,4)--(-2,4);
\draw (-6,0) node[left] {\textbf{0}};
\draw (-6,4) node[left] {\textbf{1}};
\draw (-2,4) node[right] {\textbf{2}};
\end{tikzpicture}
\caption{Symmetric chain decomposition for the polytope $L(2)$.}
\end{figure}

Coning off at $\overline{\textbf{02}}$, we get an induced geometric symmetric chain decomposition of $L(2)$ with turning sets the three line segments $\textbf{0}\text{ } \overline{\textbf{02}},\textbf{1}\text{ } \overline{\textbf{02}},\textbf{2}\text{ } \overline{\textbf{02}}$. As illustrated in Figure 5, we have the swipes $\textbf{0}\text{ }\textbf{1}\text{ }\overline{\textbf{02}}$ and $\textbf{1}\text{ }\textbf{2}\text{ }\overline{\textbf{02}}$, with directions $e_2$ and $e_1$ respectively. These swipes form a real snake, as each of the pairs $\textbf{0}$ and $\textbf{2}$, and $\overline{\textbf{02}}$ and $\overline{\textbf{02}}$ has complementary ranks.

The table below shows that the strong hyperplane condition is satisfied. For example, we see in the second line of the table that the two directions of swipes adjacent to the second turning set are $x$ and $y$, and we have the hyperplane equation $x+y=1$ with the coefficients of $x$ and $y$ both $1$. As the complexities of all turning sets are $1$, by Theorem 3.3 we get induced symmetric chain decomposition of $L(2,n)$ for all $n$.

\begin{center}
\begin{tabular}{| c | c | c | c|}
\hline
Turning Set & Complexity & Directions In & Directions Out\\
\hline
$x-y=0$ & 1 &  & y\\
$x+y=1$ & 1 & y & x\\
$x-y=0$ & 1 & x & \\
\hline
\end{tabular}
\end{center}

\end{exmp}

In the Section 4 we proceed similarly, decomposing $L(m)$ for $m=3,4,5$, and a family of polytopes derived from $L(4)$.

\section{Main Results}
In this section, we state our main results. We postpone the proofs of the general theorems to Section 5. Theorems 3.1 and 3.7 are proved in Section 4 using the general main results. Further results appear in Section 6. Recall $L(4)_{a,b,c,d}$ was defined to be $L(4)$ with the coordinate directions stretched by factors $a,b,c,d$ respectively.

\begin{thm}
The polytopes $L(m)$ have geometric symmetric chain decompositions for $m \le 5$, and they satisfy the strong hyperplane condition for $m \le 4$. If we restrict our chains to points of denominator $n$, then we obtain a symmetric chain decomposition of $L(m,n)$ for all $n$. For $m=5$, there exists a disjoint collection of symmetric chains in $L(5,27n)$ which cover $L(5,n)$.
\end{thm}

\begin{prop}
A geometric symmetric chain decomposition of $P$ induces a symmetric chain decomposition of $P$.
\end{prop}

\begin{thm}
Suppose there is a geometric symmetric chain decomposition of the polytope $P$ satisfying the strong hyperplane condition. Then there is an induced symmetric chain decomposition of the poset $P(kN)$ for all $k$, where $N$ is the least common multiple of the complexities of the turning sets.
\end{thm}

\begin{cor}
If all vertices of a polytope $P$ have integral coordinates,  and $P$ has a geometric symmetric chain decomposition satisfying the strong hyperplane conditions with the affine spans of the turning sets all passing through at least one vertex of $P$, then we get a symmetric chain decomposition of $P(k)$ for all $k$.
\end{cor}

\begin{thm}
Given a geometric symmetric chain decomposition of a polytope $P$ with the affine spans of all turning sets rational, there exists an integer $M$ such that for every $k \ge 1$, we can find a disjoint collection of symmetric chains in $P(kM)$, with each point of $P(k)$ contained in some chain.
\end{thm}

\begin{thm}
Let $P$ be a polytope of dimension $m$ in $\mathbb{R}^m$. Then, given an asymptotic geometric symmetric chain decomposition of $P$ with the weak hyperplane condition, we can delete $O(\frac{1}{k})$ of the points of the poset $P(k)$ such that we can decompose the remainder into symmetric chains.
\end{thm}

\begin{thm}
There exists a natural one-parameter deformation of $L(4)$ formed by removing an (open) simplex from $L(4)_{1,\frac{1}{t},t,1}$ to achieve volume rank symmetry, such that for each $t \ge 1$ there exists an asymptotic geometric symmetric chain decomposition with the weak hyperplane condition.
\end{thm}

\section{Geometric decompositions }
In this section, we construct various decompositions of $L(m)$ for $m=3,4,5$ and related polytopes, proving Theorems 3.1 and 3.7. The geometric decompositions we present, and in particular the coordinates of the projection points, were not found by trial and error, but were solved for at the end. At the start of Subsection 4.3 (in which Theorem 3.7 is proved), we shall discuss in detail the geometric intuition and tools we use to get our decompositions.

\subsection{Decomposing the polytope $L(3)$}
Set $\textbf{0}=(0,0,0)$, $\textbf{1}=(0,0,1)$, $\textbf{2}=(0,1,1)$, and $\textbf{3}=(1,1,1)$, and write $\overline{\textbf{ijk\ldots l}}$ for the barycenter of the face $\textbf{ijk} \ldots \textbf{l}$.

We create a geometric symmetric chain decomposition of $L(3)$ by projecting from $\overline{\textbf{0123}}$ (see Figure 6). The projected polytope is the union of the 4 facets $\textbf{012}$, $\textbf{013}$, $\textbf{023}$, and $\textbf{123}$, which form the boundary of $L(3)$. This has a geometric symmetric chain decomposition with one real snake $R$, and one fake snake $F$.

The real snake has turning sets $\textbf{0 }\overline{\textbf{02}}$, $\textbf{0 }\overline{\textbf{03}}$, $\textbf{1 }\overline{\textbf{13}}$, $\textbf{2 }\overline{\textbf{13}}$, 
$\textbf{3 }\overline{\textbf{13}}$.

The fake snake has turning sets
$\textbf{0 }\overline{\textbf{02}}$,
$\textbf{1 }\overline{\textbf{02}}$,
$\textbf{2 }\overline{\textbf{02}}$, $\textbf{3 }\overline{\textbf{03}}$, $\textbf{3 }\overline{\textbf{13}}$.

We note that each of the pairs $\textbf{0}$ and $\textbf{3}$, and $\overline{\textbf{02}}$ and $\overline{\textbf{13}}$ has complementary ranks.

Note the directions of the swipes are determined by the turning sets (e.g. the first swipe of the real snake is in the direction $e_1$).
\begin{figure}[H]
\centering
\begin{tikzpicture}
\def\ax{0}; \def\ay{8};
\def\bx{0}; \def\by{3};
\def\cx{4}; \def\cy{0};
\def\dx{11}; \def\dy{3};
\draw (\ax,\ay) node[left] {\textbf{0}};
\draw (\bx,\by) node[left] {\textbf{1}};
\draw (\cx,\cy) node[below] {\textbf{2}};
\draw (\dx,\dy) node[right] {\textbf{3}};

\foreach \k in {1,2,...,7}
{
\draw[->, shorten >=0.125cm] ({(\k/8)*(\cx/2-\ax/2)+\ax},{\ay-(\k/8)*(\ay/2-\cy/2)})--({(\k/8)*(\dx/2-\ax/2)+\ax},{\ay-(\k/8)*(\ay/2-\dy/2)});
\draw[->, shorten >=0.125cm] ({(\k/8)*(\dx/2-\ax/2)+\ax},{\ay-(\k/8)*(\ay/2-\dy/2)})--({(\k/8)*(\dx/2-\bx/2)+\bx},{\by-(\k/8)*(\by/2-\dy/2)});
\draw [->, shorten >=0.125cm] ({(\k/8)*(\dx/2-\bx/2)+\bx},{\by-(\k/8)*(\by/2-\dy/2)})--({(\k/8)*(\dx/2+\bx/2-\cx)+\cx},{(\k/8)*(\dy/2+\by/2-\cy)+\cy});
\draw [->, shorten >=0.125cm] ({(\k/8)*(\dx/2+\bx/2-\cx)+\cx},{(\k/8)*(\dy/2+\by/2-\cy)+\cy})-- ({(\k/8)*(\bx/2-\dx/2)+\dx},{\dy-(\k/8)*(\dy/2-\by/2)});
}

{
\def\k{0};
\draw[->, shorten >=0.125cm] ({(\k/8)*(\dx/2-\ax/2)+\ax},{\ay-(\k/8)*(\ay/2-\dy/2)})--({(\k/8)*(\dx/2-\bx/2)+\bx},{\by-(\k/8)*(\by/2-\dy/2)});
\draw [->, shorten >=0.125cm] ({(\k/8)*(\dx/2-\bx/2)+\bx},{\by-(\k/8)*(\by/2-\dy/2)})--({(\k/8)*(\dx/2+\bx/2-\cx)+\cx},{(\k/8)*(\dy/2+\by/2-\cy)+\cy});
\draw [->, shorten >=0.125cm] ({(\k/8)*(\dx/2+\bx/2-\cx)+\cx},{(\k/8)*(\dy/2+\by/2-\cy)+\cy})-- ({(\k/8)*(\bx/2-\dx/2)+\dx},{\dy-(\k/8)*(\dy/2-\by/2)});
}

{
\def\k{8}
\draw[->, shorten >=0.125cm] ({(\k/8)*(\cx/2-\ax/2)+\ax},{\ay-(\k/8)*(\ay/2-\cy/2)})--({(\k/8)*(\dx/2-\ax/2)+\ax},{\ay-(\k/8)*(\ay/2-\dy/2)});
\draw[->, shorten >=0.125cm] ({(\k/8)*(\dx/2-\ax/2)+\ax},{\ay-(\k/8)*(\ay/2-\dy/2)})--({(\k/8)*(\dx/2-\bx/2)+\bx},{\by-(\k/8)*(\by/2-\dy/2)});
}

\draw[gray, very thin] (\ax,\ay)--(\bx,\by)--(\cx,\cy)--(\dx,\dy)--cycle;
\draw[gray, very thin] (\bx,\by)--(\dx,\dy);
\draw[gray, very thin] (\ax,\ay)--(\cx,\cy);
\draw[gray, very thin] (\bx,\by)--(\ax/2+\cx/2,\ay/2+\cy/2);

\draw[gray,very thin] (\ax/2+\cx/2,\ay/2+\cy/2)--(\ax/2+\dx/2,\ay/2+\dy/2)--(\bx/2+\dx/2,\by/2+\dy/2)--(\cx,\cy);

\draw[very thick] (\ax,\ay)--(\ax/2+\cx/2,\ay/2+\cy/2);
\draw[very thick] (\ax,\ay)--(\ax/2+\dx/2,\ay/2+\dy/2);
\draw[very thick] (\bx,\by)--(\bx/2+\dx/2,\by/2+\dy/2);
\draw[very thick] (\cx,\cy)--(\bx/2+\dx/2,\by/2+\dy/2);
\draw[very thick] (\dx,\dy)--(\bx/2+\dx/2,\by/2+\dy/2);

\foreach \k in {1,2,...,7}
{
\draw[->, dotted, shorten >=0.125cm] ({(\k/8)*(\cx/2-\ax/2)+\ax},{\ay-(\k/8)*(\ay/2-\cy/2)})--({(\k/8)*(\ax/2+\cx/2)+\ax},{(\k/8)*(\ay/2+\cy/2-\by)+\by});
\draw[->, dotted, shorten >=0.125cm] ({(\k/8)*(\ax/2+\cx/2)+\ax},{(\k/8)*(\ay/2+\cy/2-\by)+\by})--({(\k/8)*(\ax/2-\cx/2)+\cx},{\cy-(\k/8)*(\cy/2-\ay/2)});
\draw[->, dotted, shorten >=0.125cm] ({(\k/8)*(\ax/2-\cx/2)+\cx},{\cy+(\k/8)*(\ay/2-\cy/2)})--({(\k/8)*(\ax/2-\dx/2)+\dx},{\dy+(\k/8)*(\ay/2-\dy/2)});
\draw[->, dotted, shorten >=0.125cm] ({(\k/8)*(\ax/2-\dx/2)+\dx},{\dy+(\k/8)*(\ay/2-\dy/2)})--({(\k/8)*(\bx/2-\dx/2)+\dx},{\dy+(\k/8)*(\by/2-\dy/2)});
}

{
\def\k{0};
\draw[->, dotted, shorten >=0.125cm] ({(\k/8)*(\cx/2-\ax/2)+\ax},{\ay-(\k/8)*(\ay/2-\cy/2)})--({(\k/8)*(\ax/2+\cx/2)+\ax},{(\k/8)*(\ay/2+\cy/2-\by)+\by});
\draw[->, dotted, shorten >=0.125cm] ({(\k/8)*(\ax/2+\cx/2)+\ax},{(\k/8)*(\ay/2+\cy/2-\by)+\by})--({(\k/8)*(\ax/2-\cx/2)+\cx},{\cy-(\k/8)*(\cy/2-\ay/2)});
\draw[->, dotted, shorten >=0.125cm] ({(\k/8)*(\ax/2-\cx/2)+\cx},{\cy+(\k/8)*(\ay/2-\cy/2)})--({(\k/8)*(\ax/2-\dx/2)+\dx},{\dy+(\k/8)*(\ay/2-\dy/2)});
}

{
\def\k{8}
\draw[->, dotted, shorten >=0.125cm] ({(\k/8)*(\ax/2-\cx/2)+\cx},{\cy+(\k/8)*(\ay/2-\cy/2)})--({(\k/8)*(\ax/2-\dx/2)+\dx},{\dy+(\k/8)*(\ay/2-\dy/2)});
\draw[->, dotted, shorten >=0.125cm] ({(\k/8)*(\ax/2-\dx/2)+\dx},{\dy+(\k/8)*(\ay/2-\dy/2)})--({(\k/8)*(\bx/2-\dx/2)+\dx},{\dy+(\k/8)*(\by/2-\dy/2)});
}

\end{tikzpicture}
\caption{Geometric symmetric chain decomposition for the projected polytope of $L(3)$.}
\end{figure}
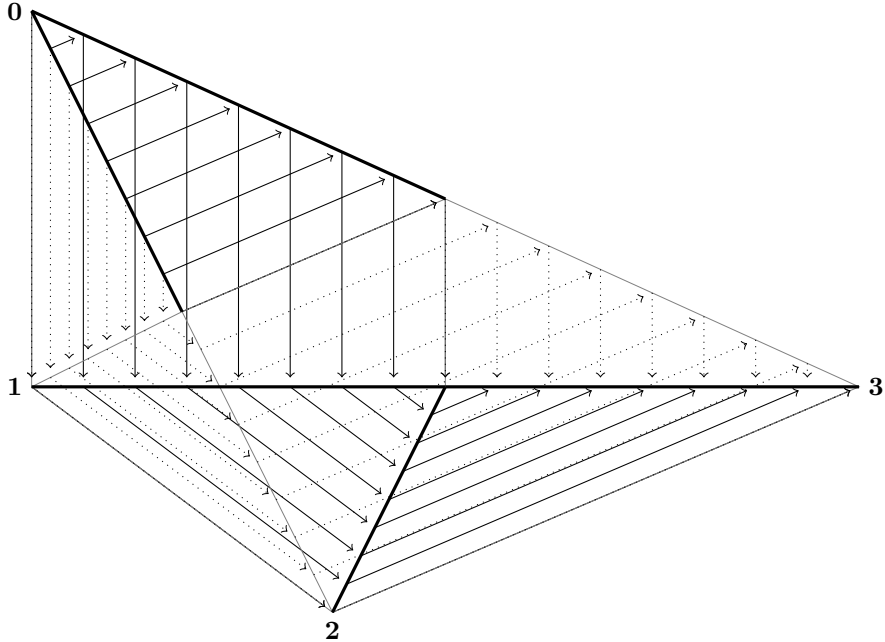

The table below shows that the strong hyperplane condition is satisfied. Therefore we get a symmetric chain decomposition of $L(3,n)$ for all $n$ by Corollary 3.4.

\begin{center}
\begin{tabular}{| c | c | c | c | c| c |}
\hline
Turning Set & \# (Real)& \# (Fake)& Complexity & Directions In & Directions Out\\
\hline
$x+y-z=0$ &  & 0 &1 & & z\\
$-x+y+z=1$ &  & 1 & 1 & z & y\\
$x+y-z=0$ & 0 & 2 &1 & y & x\\
$x-2y+z=0$ & 1 & 3 & 1 & x & z\\
$-x+y+z=1$ & 2 & 4 & 1 & z & y\\
$x+y-z=0$ & 3 &  &1 & y & x\\
$-x+y+z=1$ & 4 &  & 1 & x & \\
\hline
\end{tabular}
\end{center}

This decomposition happens to be the same as the one found in \cite{xiangdong} by computer search (there, the decomposition and verification were done by computer). 

\subsection{Decomposing the polytope $L(4)$}
Set $\textbf{0}=(0,0,0,0)$, $\textbf{1}=(0,0,0,1)$, $\textbf{2}=(0,0,1,1)$, $\textbf{3}=(0,1,1,1)$,
$\textbf{4}=(1,1,1,1)$, and write $\overline{\textbf{ijk\ldots l}}$ for the barycenter of the face $\textbf{ijk}\ldots \textbf{l}$.

We create a geometric symmetric chain decomposition of $L(4)$ by projecting from the points $\overline{\textbf{123}}$ and $\overline{\textbf{04}}$ (see Figure 7). The projected polytope is the union of the 6 faces $\textbf{012}$, $\textbf{013}$, $\textbf{023}$, $\textbf{124}$, $\textbf{134}$, $\textbf{234}$. This has a geometric symmetric chain decomposition with one real snake $R$, and one fake snake $F$.

The real snake has turning sets $\textbf{0 }\overline{\textbf{02}}$, $\textbf{0 }\overline{\textbf{03}}$, $\textbf{1 }\overline{\textbf{13}}$, $\textbf{1 }\overline{\textbf{14}}$, 
$\textbf{2 }\overline{\textbf{24}}$
, 
$\textbf{3 }\overline{\textbf{24}}$
, 
$\textbf{4 }\overline{\textbf{24}}$.

The fake snake has turning sets
$\textbf{0 }\overline{\textbf{02}}$, $\textbf{1 }\overline{\textbf{02}}$, $\textbf{2 }\overline{\textbf{02}}$, $\textbf{3 }\overline{\textbf{03}}$, 
$\textbf{3 }\overline{\textbf{13}}$
, 
$\textbf{4 }\overline{\textbf{14}}$
, 
$\textbf{4 }\overline{\textbf{24}}$.

We note that each of the pairs $\textbf{0}$ and $\textbf{4}$, and $\overline{\textbf{02}}$ and $\overline{\textbf{24}}$ has complementary ranks.

\begin{figure}[H]
\centering
\begin{tikzpicture}[scale=0.87]
\def\ax{0}; \def\ay{8};
\def\bx{0}; \def\by{3};
\def\cx{4}; \def\cy{0};
\def\dx{11}; \def\dy{3};
\def\ex{4}; \def\ey{-5};
\draw (\ax,\ay) node[left] {\textbf{0}};
\draw (\bx,\by) node[left] {\textbf{1}};
\draw (\cx,\cy) node[above] {\textbf{2}};
\draw (\dx,\dy) node[right] {\textbf{3}};
\draw (\ex,\ey) node[below] {\textbf{4}};

\foreach \k in {1,2,...,8}
{
\draw[->, shorten >=0.125cm] ({(\k/8)*(\cx/2-\ax/2)+\ax},{\ay-(\k/8)*(\ay/2-\cy/2)})--({(\k/8)*(\dx/2-\ax/2)+\ax},{\ay-(\k/8)*(\ay/2-\dy/2)});
\draw[->, shorten >=0.125cm] ({(\k/8)*(\dx/2-\ax/2)+\ax},{\ay-(\k/8)*(\ay/2-\dy/2)})--({(\k/8)*(\dx/2-\bx/2)+\bx},{\by-(\k/8)*(\by/2-\dy/2)});
}

{
\def\k{0};
\draw[->, shorten >=0.125cm] ({(\k/8)*(\dx/2-\ax/2)+\ax},{\ay-(\k/8)*(\ay/2-\dy/2)})--({(\k/8)*(\dx/2-\bx/2)+\bx},{\by-(\k/8)*(\by/2-\dy/2)});
}

\draw[gray, very thin] (\ax,\ay)--(\bx,\by)--(\cx,\cy)--(\dx,\dy)--cycle;
\draw[gray, very thin] (\bx,\by)--(\dx,\dy);
\draw[gray, very thin] (\ax,\ay)--(\cx,\cy);

\draw[gray,very thin] (\ax/2+\cx/2,\ay/2+\cy/2)--(\ax/2+\dx/2,\ay/2+\dy/2)--(\bx/2+\dx/2,\by/2+\dy/2);

\draw[very thick] (\ax,\ay)--(\ax/2+\cx/2,\ay/2+\cy/2);
\draw[very thick] (\ax,\ay)--(\ax/2+\dx/2,\ay/2+\dy/2);
\draw[very thick] (\bx,\by)--(\bx/2+\dx/2,\by/2+\dy/2);
\draw[very thick] (\bx,\by)--(\bx/2+\ex/2,\by/2+\ey/2);

\foreach \k in {1,2,...,7}
{
\draw[->, dotted, shorten >=0.125cm] ({(\k/8)*(\cx/2-\ax/2)+\ax},{\ay-(\k/8)*(\ay/2-\cy/2)})--({(\k/8)*(\ax/2+\cx/2-\bx)+\bx},{(\k/8)*(\ay/2+\cy/2-\by)+\by});
\draw[->, dotted, shorten >=0.125cm] ({(\k/8)*(\ax/2+\cx/2-\bx)+\bx},{(\k/8)*(\ay/2+\cy/2-\by)+\by})--({(\k/8)*(\ax/2-\cx/2)+\cx},{\cy-(\k/8)*(\cy/2-\ay/2)});
\draw[->, dotted, shorten >=0.125cm] ({(\k/8)*(\ax/2-\cx/2)+\cx},{\cy+(\k/8)*(\ay/2-\cy/2)})--({(\k/8)*(\ax/2-\dx/2)+\dx},{\dy+(\k/8)*(\ay/2-\dy/2)});
\draw[->, dotted, shorten >=0.125cm] ({(\k/8)*(\ax/2-\dx/2)+\dx},{\dy+(\k/8)*(\ay/2-\dy/2)})--({(\k/8)*(\bx/2-\dx/2)+\dx},{\dy+(\k/8)*(\by/2-\dy/2)});
}

{
\def\k{0};
\draw[->, dotted, shorten >=0.125cm] ({(\k/8)*(\cx/2-\ax/2)+\ax},{\ay-(\k/8)*(\ay/2-\cy/2)})--({(\k/8)*(\ax/2+\cx/2-\bx)+\bx},{(\k/8)*(\ay/2+\cy/2-\by)+\by});
\draw[->, dotted, shorten >=0.125cm] ({(\k/8)*(\ax/2+\cx/2-\bx)+\bx},{(\k/8)*(\ay/2+\cy/2-\by)+\by})--({(\k/8)*(\ax/2-\cx/2)+\cx},{\cy-(\k/8)*(\cy/2-\ay/2)});
\draw[->, dotted, shorten >=0.125cm] ({(\k/8)*(\ax/2-\cx/2)+\cx},{\cy+(\k/8)*(\ay/2-\cy/2)})--({(\k/8)*(\ax/2-\dx/2)+\dx},{\dy+(\k/8)*(\ay/2-\dy/2)});
}

{
\def\k{8};
\draw[->, dotted, shorten >=0.125cm] ({(\k/8)*(\ax/2-\cx/2)+\cx},{\cy+(\k/8)*(\ay/2-\cy/2)})--({(\k/8)*(\ax/2-\dx/2)+\dx},{\dy+(\k/8)*(\ay/2-\dy/2)});
\draw[->, dotted, shorten >=0.125cm] ({(\k/8)*(\ax/2-\dx/2)+\dx},{\dy+(\k/8)*(\ay/2-\dy/2)})--({(\k/8)*(\bx/2-\dx/2)+\dx},{\dy+(\k/8)*(\by/2-\dy/2)});
}

\draw[gray, very thin] (\ex,\ey)--(\bx,\by)--(\cx,\cy)--(\dx,\dy)--cycle;
\draw[gray, very thin] (\ex,\ey)--(\cx,\cy);

\draw[very thick] (\dx,\dy)--(\ex/2+\cx/2,\ey/2+\cy/2);
\draw[very thick] (\ex,\ey)--(\cx,\cy);

\foreach \k in {1,2,...,8}
{
\draw[<-,dotted, shorten <=0.125cm] ({(\k/8)*(\cx/2-\ex/2)+\ex},{\ey-(\k/8)*(\ey/2-\cy/2)})--({(\k/8)*(\bx/2-\ex/2)+\ex},{\ey-(\k/8)*(\ey/2-\by/2)});
\draw[<-,dotted, shorten <=0.125cm] ({(\k/8)*(\bx/2-\ex/2)+\ex},{\ey-(\k/8)*(\ey/2-\by/2)})--({(\k/8)*(\bx/2-\dx/2)+\dx},{\dy-(\k/8)*(\dy/2-\by/2)});
}

{
\def\k{0};
\draw[<-,dotted, shorten <=0.125cm] ({(\k/8)*(\bx/2-\ex/2)+\ex},{\ey-(\k/8)*(\ey/2-\by/2)})--({(\k/8)*(\bx/2-\dx/2)+\dx},{\dy-(\k/8)*(\dy/2-\by/2)});
}

\foreach \k in {1,2,...,7}
{
\draw[<-, shorten <=0.125cm] ({(\k/8)*(\cx/2-\ex/2)+\ex},{\ey-(\k/8)*(\ey/2-\cy/2)})--({(\k/8)*(\ex/2+\cx/2-\dx)+\dx},{(\k/8)*(\ey/2+\cy/2-\dy)+\dy});
\draw[<-, shorten <=0.125cm] ({(\k/8)*(\ex/2+\cx/2-\dx)+\dx},{(\k/8)*(\ey/2+\cy/2-\dy)+\dy})--({(\k/8)*(\ex/2-\cx/2)+\cx},{\cy-(\k/8)*(\cy/2-\ey/2)});
\draw[<-, shorten <=0.125cm] ({(\k/8)*(\ex/2-\cx/2)+\cx},{\cy+(\k/8)*(\ey/2-\cy/2)})--({(\k/8)*(\ex/2-\bx/2)+\bx},{\by+(\k/8)*(\ey/2-\by/2)});
\draw[<-, shorten <=0.125cm] ({(\k/8)*(\ex/2-\bx/2)+\bx},{\by+(\k/8)*(\ey/2-\by/2)})--({(\k/8)*(\dx/2-\bx/2)+\bx},{\by+(\k/8)*(\dy/2-\by/2)});
}

{
\def\k{0};
\draw[<-, shorten <=0.125cm] ({(\k/8)*(\cx/2-\ex/2)+\ex},{\ey-(\k/8)*(\ey/2-\cy/2)})--({(\k/8)*(\ex/2+\cx/2-\dx)+\dx},{(\k/8)*(\ey/2+\cy/2-\dy)+\dy});
\draw[<-, shorten <=0.125cm] ({(\k/8)*(\ex/2+\cx/2-\dx)+\dx},{(\k/8)*(\ey/2+\cy/2-\dy)+\dy})--({(\k/8)*(\ex/2-\cx/2)+\cx},{\cy-(\k/8)*(\cy/2-\ey/2)});
\draw[<-, shorten <=0.125cm] ({(\k/8)*(\ex/2-\cx/2)+\cx},{\cy+(\k/8)*(\ey/2-\cy/2)})--({(\k/8)*(\ex/2-\bx/2)+\bx},{\by+(\k/8)*(\ey/2-\by/2)});
}

{
\def\k{8};
\draw[<-, shorten <=0.125cm] ({(\k/8)*(\ex/2-\cx/2)+\cx},{\cy+(\k/8)*(\ey/2-\cy/2)})--({(\k/8)*(\ex/2-\bx/2)+\bx},{\by+(\k/8)*(\ey/2-\by/2)});
\draw[<-, shorten <=0.125cm] ({(\k/8)*(\ex/2-\bx/2)+\bx},{\by+(\k/8)*(\ey/2-\by/2)})--({(\k/8)*(\dx/2-\bx/2)+\bx},{\by+(\k/8)*(\dy/2-\by/2)});
}

\draw[gray, very thin] (\bx,\by)--(\ax/2+\cx/2,\ay/2+\cy/2);

\end{tikzpicture}
\caption{Geometric symmetric chain decomposition for the projected polytope of $L(4)$.}
\end{figure}
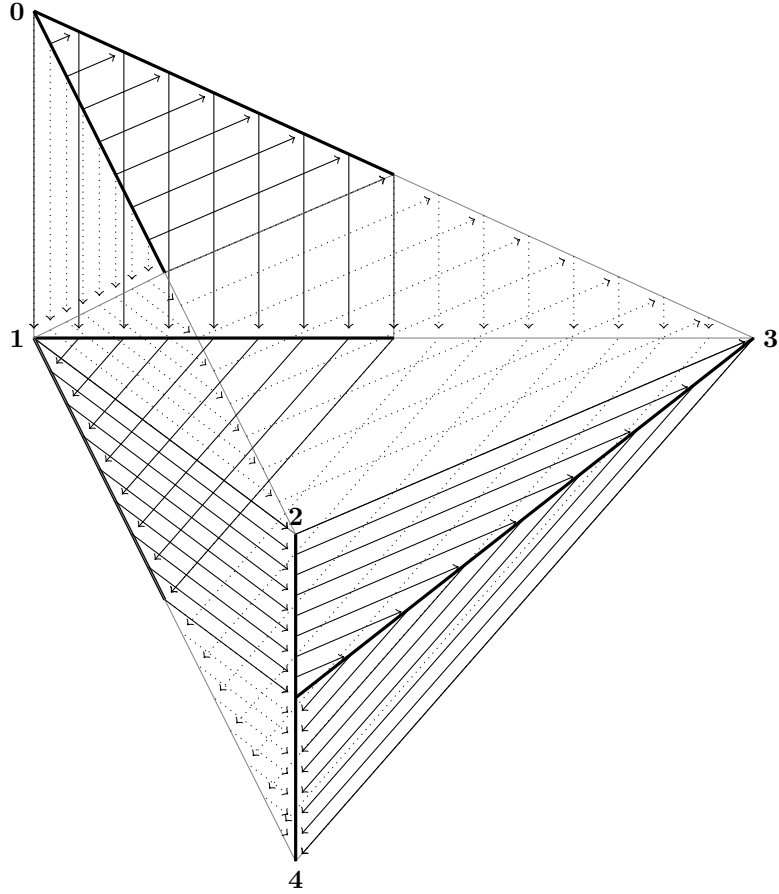

The table below shows that the strong hyperplane condition is satisfied. Therefore we get a symmetric chain decomposition of $L(4,n)$ for all $n$ by Corollary 3.4.

\begin{center}
\begin{tabular}{| c | c | c | c | c| c |}
\hline
Turning Set & \# (Real)& \# (Fake)& Complexity & Directions In & Directions Out\\
\hline
$-x+y+z-w=0$ & 0,4,6& 0,2,6&1 & x,z & y,w \\
$2x-2y+z+w=1$ &  &  1&1  & w & z\\
$y-2z+w=0$ &  1&  3&1  & y & w\\
$x+w=1$ & 2& 4&1 & w & x\\
$x-2y+z=0$ &  3&  5&1  & x & z\\
$x+y-2z+2w=1$ & 5 & & 1&  y& x\\
\hline
\end{tabular}
\end{center}

This decomposition happens to be the one found in \cite{west}, where the method of finding the decomposition was not stated. Note that from our perspective, we clearly see the similarity of the decompositions of $L(3)$ and $L(4)$ coming from the heurstically similar natures of the decompositions of the associated two dimensional projected polytopes. Indeed, each two dimensional snake in the projected polytope of $L(4)$ is formed by gluing together two truncated snakes from the projected polytope of $L(3)$ (so using our method, one is inspired to find one of the decompositions from the other one, as we did).

\begin{rmk}
If we project from the points $\overline{\textbf{13}}$ and $\overline{\textbf{024}}$ instead, then there exists a similar decomposition of the projected polytope using again one real and one fake snake starting at the line segment $\textbf{0}$ $\overline{\textbf{02}}$ (see Figure 8). The table below shows that the strong hyperplane condition is satisfied, and so by Corollary 3.4 we get an alternative symmetric chain decomposition of the poset $L(4,n)$ as was found in \cite{xiangdong} by computer methods.
\end{rmk}

\begin{center}
\begin{tabular}{| c | c | c | c | c| c |}
\hline
Turning Set & \# (Real)& \# (Fake)& Complexity & Directions In & Directions Out\\
\hline
$-x+y+z-w=0$ & 0,2,4,6&0,2,4,6&1 &  x,z& y,w \\
$-y+z+w=1$ &  3&1,5  &1  & w & z\\
$x+y-z=0$ &  1,5&3  &1  & y & x\\
\hline
\end{tabular}
\end{center}

\begin{figure}[H]
\centering
\begin{tikzpicture}[scale=0.82]
\def\bx{1}; \def\by{8};
\def\ax{0}; \def\ay{3};
\def\cx{4.4}; \def\cy{1};
\def\ex{11}; \def\ey{3};
\def\dx{4.4}; \def\dy{-5};
\draw (\ax,\ay) node[left] {\textbf{0}};
\draw (\bx,\by) node[left] {\textbf{1}};
\draw (\cx,\cy) node[above] {\textbf{2}};
\draw (\dx,\dy) node[right] {\textbf{3}};
\draw (\ex,\ey) node[below] {\textbf{4}};

\draw[gray, very thin] (\ax,\ay)--(\ex,\ey);

\foreach \k in {1,2,...,8}
{
\draw[->, shorten >=0.125cm] ({(\k/8)*(\cx/2-\ax/2)+\ax},{\ay-(\k/8)*(\ay/2-\cy/2)})--({(\k/8)*(\dx/2-\ax/2)+\ax},{\ay-(\k/8)*(\ay/2-\dy/2)});
\draw[->, shorten >=0.125cm] ({(\k/8)*(\dx/2-\ax/2)+\ax},{\ay-(\k/8)*(\ay/2-\dy/2)})--({(\k/8)*(\ex/2-\ax/2)+\ax},{\ay-(\k/8)*(\ay/2-\ey/2)});
}

\draw[gray, very thin] (\ax,\ay)--(\bx,\by)--(\cx,\cy)--(\dx,\dy)--cycle;
\draw[gray, very thin] (\ax,\ay)--(\cx,\cy);

\draw[very thick] (\ax,\ay)--(\ax/2+\cx/2,\ay/2+\cy/2);
\draw[very thick] (\ax,\ay)--(\ax/2+\dx/2,\ay/2+\dy/2);
\draw[very thick] (\ax,\ay)--(\ax/2+\ex/2,\ay/2+\ey/2);
\draw[very thick] (\bx,\by)--(\bx/2+\ex/2,\by/2+\ey/2);

\foreach \k in {1,2,...,7}
{
\draw[->, dotted, shorten >=0.125cm] ({(\k/8)*(\cx/2-\ax/2)+\ax},{\ay-(\k/8)*(\ay/2-\cy/2)})--({(\k/8)*(\ax/2+\cx/2-\bx)+\bx},{(\k/8)*(\ay/2+\cy/2-\by)+\by});
\draw[->, dotted, shorten >=0.125cm] ({(\k/8)*(\ax/2+\cx/2-\bx)+\bx},{(\k/8)*(\ay/2+\cy/2-\by)+\by})--({(\k/8)*(\ax/2-\cx/2)+\cx},{\cy-(\k/8)*(\cy/2-\ay/2)});
\draw[->, dotted, shorten >=0.125cm] ({(\k/8)*(\ax/2-\cx/2)+\cx},{\cy+(\k/8)*(\ay/2-\cy/2)})--({(\k/8)*(\ax/2-\dx/2)+\dx},{\dy+(\k/8)*(\ay/2-\dy/2)});
\draw[->, dotted, shorten >=0.125cm] ({(\k/8)*(\ax/2-\dx/2)+\dx},{\dy+(\k/8)*(\ay/2-\dy/2)})--({(\k/8)*(\ax/2-\ex/2)+\ex},{\ey+(\k/8)*(\ay/2-\ey/2)});
}

{
\def\k{0};
\draw[->, dotted, shorten >=0.125cm] ({(\k/8)*(\cx/2-\ax/2)+\ax},{\ay-(\k/8)*(\ay/2-\cy/2)})--({(\k/8)*(\ax/2+\cx/2-\bx)+\bx},{(\k/8)*(\ay/2+\cy/2-\by)+\by});
\draw[->, dotted, shorten >=0.125cm] ({(\k/8)*(\ax/2+\cx/2-\bx)+\bx},{(\k/8)*(\ay/2+\cy/2-\by)+\by})--({(\k/8)*(\ax/2-\cx/2)+\cx},{\cy-(\k/8)*(\cy/2-\ay/2)});
\draw[->, dotted, shorten >=0.125cm] ({(\k/8)*(\ax/2-\cx/2)+\cx},{\cy+(\k/8)*(\ay/2-\cy/2)})--({(\k/8)*(\ax/2-\dx/2)+\dx},{\dy+(\k/8)*(\ay/2-\dy/2)});
}

\draw[gray, very thin] (\ex,\ey)--(\bx,\by)--(\cx,\cy)--(\dx,\dy)--cycle;
\draw[gray, very thin] (\ex,\ey)--(\cx,\cy);

\draw[very thick] (\dx,\dy)--(\ex/2+\cx/2,\ey/2+\cy/2);
\draw[very thick] (\ex,\ey)--(\cx,\cy);

\foreach \k in {1,2,...,8}
{
\draw[<-,dotted, shorten <=0.125cm] ({(\k/8)*(\cx/2-\ex/2)+\ex},{\ey-(\k/8)*(\ey/2-\cy/2)})--({(\k/8)*(\bx/2-\ex/2)+\ex},{\ey-(\k/8)*(\ey/2-\by/2)});
\draw[<-,dotted, shorten <=0.125cm] ({(\k/8)*(\bx/2-\ex/2)+\ex},{\ey-(\k/8)*(\ey/2-\by/2)})--({(\k/8)*(\ax/2-\ex/2)+\ex},{\ey-(\k/8)*(\ey/2-\ay/2)});
}

{
\def\k{0};
\draw[<-,dotted, shorten <=0.125cm] ({(\k/8)*(\bx/2-\ex/2)+\ex},{\ey-(\k/8)*(\ey/2-\by/2)})--({(\k/8)*(\bx/2-\dx/2)+\dx},{\dy-(\k/8)*(\dy/2-\by/2)});
}

\foreach \k in {1,2,...,7}
{
\draw[<-, shorten <=0.125cm] ({(\k/8)*(\cx/2-\ex/2)+\ex},{\ey-(\k/8)*(\ey/2-\cy/2)})--({(\k/8)*(\ex/2+\cx/2-\dx)+\dx},{(\k/8)*(\ey/2+\cy/2-\dy)+\dy});
\draw[<-, shorten <=0.125cm] ({(\k/8)*(\ex/2+\cx/2-\dx)+\dx},{(\k/8)*(\ey/2+\cy/2-\dy)+\dy})--({(\k/8)*(\ex/2-\cx/2)+\cx},{\cy-(\k/8)*(\cy/2-\ey/2)});
\draw[<-, shorten <=0.125cm] ({(\k/8)*(\ex/2-\cx/2)+\cx},{\cy+(\k/8)*(\ey/2-\cy/2)})--({(\k/8)*(\ex/2-\bx/2)+\bx},{\by+(\k/8)*(\ey/2-\by/2)});
\draw[<-, shorten <=0.125cm] ({(\k/8)*(\ex/2-\bx/2)+\bx},{\by+(\k/8)*(\ey/2-\by/2)})--({(\k/8)*(\ex/2-\ax/2)+\ax},{\ay+(\k/8)*(\ey/2-\ay/2)});
}

{
\def\k{0};
\draw[<-, shorten <=0.125cm] ({(\k/8)*(\cx/2-\ex/2)+\ex},{\ey-(\k/8)*(\ey/2-\cy/2)})--({(\k/8)*(\ex/2+\cx/2-\dx)+\dx},{(\k/8)*(\ey/2+\cy/2-\dy)+\dy});
\draw[<-, shorten <=0.125cm] ({(\k/8)*(\ex/2+\cx/2-\dx)+\dx},{(\k/8)*(\ey/2+\cy/2-\dy)+\dy})--({(\k/8)*(\ex/2-\cx/2)+\cx},{\cy-(\k/8)*(\cy/2-\ey/2)});
\draw[<-, shorten <=0.125cm] ({(\k/8)*(\ex/2-\cx/2)+\cx},{\cy+(\k/8)*(\ey/2-\cy/2)})--({(\k/8)*(\ex/2-\bx/2)+\bx},{\by+(\k/8)*(\ey/2-\by/2)});
}

{
\def\k{8};
\draw[<-, shorten <=0.125cm] ({(\k/8)*(\ex/2-\cx/2)+\cx},{\cy+(\k/8)*(\ey/2-\cy/2)})--({(\k/8)*(\ex/2-\bx/2)+\bx},{\by+(\k/8)*(\ey/2-\by/2)});
\draw[<-, shorten <=0.125cm] ({(\k/8)*(\ex/2-\bx/2)+\bx},{\by+(\k/8)*(\ey/2-\by/2)})--({(\k/8)*(\ex/2-\ax/2)+\ax},{\ay+(\k/8)*(\ey/2-\ay/2)});
}

\draw[gray, very thin] (\bx,\by)--(\ax/2+\cx/2,\ay/2+\cy/2);

\end{tikzpicture}
\caption{Alternative geometric symmetric chain decomposition for $L(4)$.}
\end{figure}
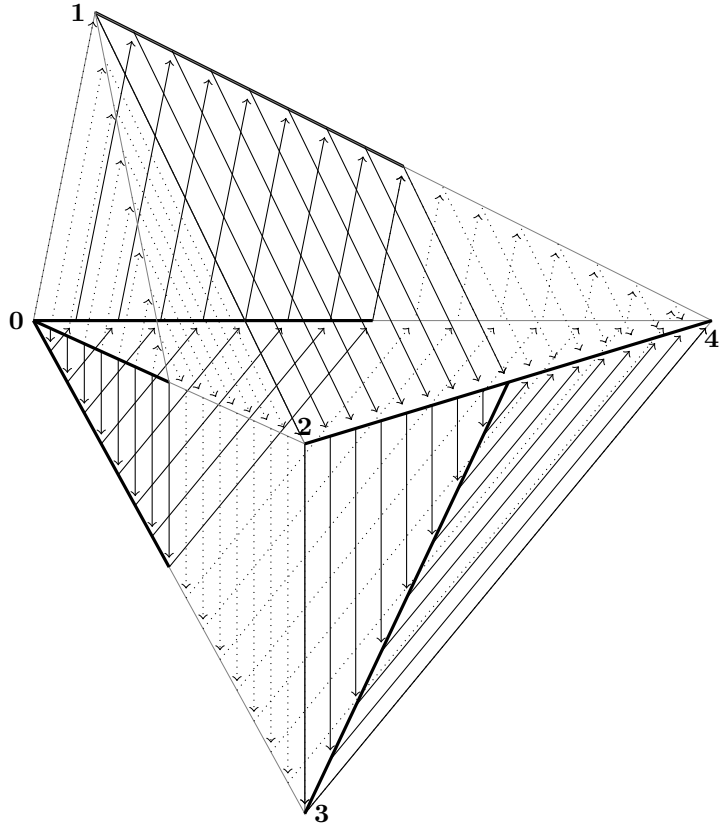

\subsection{Investigating the polytope $L(4)_{a,b,c,d}$}
Motivated by the geometric description above, we investigate the polytope $L(4)_{a,b,c,d}$, obtained by ``stretching'' the 4-dimensional polytope $L(4)$ in the coordinate directions by scaling the axes by $a,b,c,d$ respectively. As it turns out later, one necessary condition for $L(4)_{a,b,c,d}$ to admit an asymptotic geometric symmetric chain decomposition with the weak hyperplane condition is that $a=d$ and $b=c$ (otherwise, it is not even possible to get asymptotic decompositions for the posets $L(4,k)_{a,b,c,d}$). However, as we shall see, if we remove a suitable region to restore rank-symmetry from a certain one-parameter family of $L(4)_{a,b,c,d}$ specializing to $L(4)$, but not satisfying $b=c$ in general, we still get an asymptotic symmetric geometric chain decomposition with the weak hyperplane condition. This produces asymptotic decompositions of the underlying posets. 

The strong hyperplane condition involves a divisibility condition on the coefficients of the hyperplane containing the turning set, whereas the weak hyperplane condition can be expressed topologically by requiring that a certain direction is contained in the hyperplane. Consequently, we first aim to satisfy the weak hyperplane condition, and only after do we check if the strong hyperplane condition is satisfied.

We remark that for a turning set of both a swipe of direction $i$ and a swipe of direction $j$ ($i \ne j$), the weak hyperplane condition is equivalent to saying that the direction vector $e_i-e_j$ lies inside the turning set.

The strategy will be to use the geometric decomposition of the two dimensional projected polytope from Figure 7 as a skeleton of the decomposition that we seek, varying the vertices of the turning sets and in particular the coning off points, while trying to keep them on the same faces, seeking to ensure that the symmetry condition and the geometric condition from the previous paragraph are attained. The construction involves coning off the two dimensional surface at two points which will be constrained by the aforementioned conditions.

By abuse of notation, we use the previous bolded notation from Subsection 4.2 for both $L(4)$ and $L(4)_{a,b,c,d}$. One of the coning off points $\textbf{B}$ was the midpoint of $\textbf{04}$, and we constrain it to remain on this line segment. As it must have middle rank, it remains the midpoint of $\textbf{04}$. The other coning off point $\textbf{B'}$ was the barycenter of $\textbf{123}$, and we constrain it to lie on one of the two faces $\textbf{0123}$ and $\textbf{1234}$ on the boundary of $L(4)_{a,b,c,d}$ which do not include $\textbf{B}$.

What follows now is an investigation of the constraints on $a,b,c,d$ and a geometric construction of the desired point $\textbf{B'}$. As $\textbf{B'}$ might not lie on $\textbf{123}$, the resulting coned off polytope will miss a region of $L(4)_{a,b,c,d}$. We remark that we could have instead used linear algebra (possibly aided by computer algebra software) to produce the construction below in a rather straightforward way, but the geometric construction motivates a priori why such a construction is possible.

In the projected polytope of $L(4)$ in Figure 7, the second turning set of the fake snake was the line segment $\textbf{1p}$, and the second last turning set of the real snake was the line segment $\textbf{3q}$. We constrain $\textbf{p}$ and $\textbf{q}$ in $L(4)_{a,b,c,d}$ to lie on the line segments $\textbf{02}$ and $\textbf{24}$ respectively.

We now cone off the two dimensional surface at $\textbf{B}$ and $\textbf{B'}$. For the turning set $\textbf{1pBB'}$ to satisfy the weak hyperplane condition, $\textbf{1pBB'}$ must contain the direction vector $e_3-e_4$. This is equivalent to saying that $\textbf{1p}$ is the angle bisector of the right angle $\angle \textbf{012}$, or (see \cite{Johnson}) that $|\textbf{01}|/|\textbf{12}|=|\textbf{0p}|/|\textbf{p2}|$. Similarly, $\textbf{3q}$ is the angle bisector of the right angle $\angle\textbf{234}$, or equivalently $|\textbf{43}|/|\textbf{32}|=|\textbf{4q}|/|\textbf{q2}|$. By parallelism, the chain through $\textbf{p},\textbf{q}$ divides the line segments $\textbf{02},\textbf{03},\textbf{13},\textbf{14},\textbf{24}$ in the same ratio. Therefore $|\textbf{01}|/|\textbf{12}|=|\textbf{0p}|/|\textbf{p2}|=|\textbf{q2}|/|\textbf{4q}|=|\textbf{32}|/|\textbf{43}|$ (i.e. $ad=bc$).

As $\textbf{1p}$ is parallel to $e_3-e_4$, the rank of $\textbf{p}$ is the same as the rank of $\textbf{1}$. Similarly, the rank of $\textbf{q}$ is the same as the rank of $\textbf{3}$. As there is a symmetric chain with endpoints $\textbf{p}$ and $\textbf{q}$, we get the sum of the ranks of $\textbf{1}$ and $\textbf{3}$ is the same as the sum of the ranks of $\textbf{0}$ and $\textbf{4}$, which implies $|\textbf{01}|=|\textbf{34}|$ (i.e. $a=d$). Because of the relations $ad=bc$ and $a=d$, we restrict ourselves to the family of polytopes $L(4)_t=L(4)_{1,\frac{1}{t},t,1}$ for $t\ge 1$ ($0<t\le 1$ is identical by symmetry).

We now investigate the constraints on $\textbf{B'}$. There are up to $7$ distinct hyperplanes on which the turning sets lie. These are: $\textbf{1pBB'}$, $\textbf{02BB'}$, $\textbf{03BB'}$, $\textbf{13BB'}$, $\textbf{14BB'}$, $\textbf{24BB'}$, $\textbf{3qBB'}$, which must contain the direction vectors $e_4-e_3$, $e_3-e_2$, $e_2-e_4$, $e_4-e_1$, $e_1-e_3$, $e_3-e_2$, and $e_2-e_1$ respectively in order to satisfy the weak hyperplane condition. We also require that $\textbf{B'}$ has middle rank. Finally, as it turns out, since $t \ge 1$ we shall construct $\textbf{B'}$ to lie on the face $\textbf{0123}$ (if we had taken $0<t\le 1$, we would have constructed $\textbf{B'}$ to lie on $\textbf{1234}$).

Note that a two dimensional plane in $\mathbb{R}^4$ and a direction vector not parallel to the plane determine a hyperplane. The first and last hyperplane impose no constraint on $\textbf{B'}$ as $\textbf{1p}$ and $\textbf{3q}$ are parallel to $e_4-e_3$ and $e_2-e_1$ respectively. The hyperplane $\textbf{02BB'}$ is the same as the hyperplane $\textbf{24BB'}$, because $\textbf{B}$ lies on $\textbf{04}$. Also, the weak hyperplane condition for $\textbf{13BB'}$ is automatically satisfied, as the vector connecting the midpoint of $\textbf{04}$ (i.e. $\textbf{B}$) to the midpoint of $\textbf{13}$ is $\frac{1}{2}(e_4-e_1)$.

Therefore, we only have to consider the weak hyperplane conditions for $\textbf{14BB'}$, $\textbf{02BB'}$, and $\textbf{03BB'}$. As $\textbf{B}$ is the midpoint of $\textbf{04}$, these hyperplanes are equal to $\textbf{014B'}$, $\textbf{024B'}$, and $\textbf{034B'}$ respectively.

Define the hyperplanes $H_1,H_2,H_3$ as follows:
\begin{enumerate}
\item $H_1$ is generated by $\textbf{014}$ and direction vector $e_1-e_3$.
\item $H_2$ is generated by $\textbf{024}$ and direction vector $e_2-e_3$.
\item $H_3$ is generated by $\textbf{034}$ and direction vector $e_2-e_4$.
\end{enumerate}
The three hyperplane conditions above are equivalent to $\textbf{B'}\in H_1 \cap H_2 \cap H_3$.

We shall now construct by a purely geometric argument a point $\textbf{B'} \in H_1 \cap H_2 \cap H_3$ of middle rank lying inside $\textbf{0123}$.

\begin{figure}[H]
\centering
\begin{tikzpicture}[scale=0.82]
\def\ax{0}; \def\ay{8};
\def\bx{0}; \def\by{3};
\def\cx{7}; \def\cy{0};
\def\dx{11}; \def\dy{3};
\draw (\ax,\ay) node[left] {\textbf{0}};
\draw (\bx,\by) node[left] {\textbf{1}};
\draw (\cx,\cy) node[below] {\textbf{2}};
\draw (\dx,\dy) node[right] {\textbf{3}};
\draw (\ax,\ay)--(\bx,\by);
\draw (\bx,\by)--node[below] {$\textbf{r}_3$} (\cx,\cy);
\draw (\cx,\cy)--node[below] {$\textbf{r}_1$} (\dx,\dy);
\draw (\ax,\ay)--(\dx,\dy);
\draw (\bx,\by)-- node[above] {$\textbf{r}_2$}(\dx,\dy);
\draw (\ax,\ay)--(\cx,\cy);
\draw[dotted] (\cx,\cy)--(\bx/2+\dx/2,\by/2+\dy/2);
\draw[dotted] (\dx,\dy)--(\bx/2+\cx/2,\by/2+\cy/2);
\draw[dotted] (\bx,\by)--(\dx/2+\cx/2,\dy/2+\cy/2);
\draw[dotted] (\ax,\ay)--(\bx/2+\cx/2,\by/2+\cy/2);
\draw[dotted] (\ax,\ay)--(\bx/2+\dx/2,\by/2+\dy/2);
\draw[dotted] (\ax,\ay)--(\dx/2+\cx/2,\dy/2+\cy/2);
\draw[dotted] (\ax,\ay)--(\bx/3+\cx/3+\dx/3,\by/3+\cy/3+\dy/3) node[right] {$\textbf{r}$} ;
\end{tikzpicture}
\caption{Hyperplanes intersecting inside $\textbf{0123}$}
\end{figure}
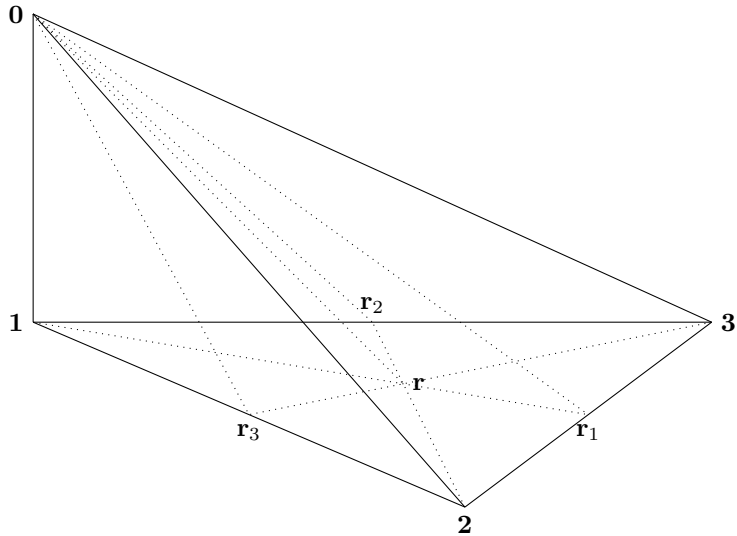

\begin{lem}
In the tetrahedron $\textbf{0123}$, the hyperplane $H_3$ intersects the line segment $\textbf{12}$ in a point $\textbf{r}_3$ such that $|\textbf{1}\textbf{r}_3|/|\textbf{r}_3\textbf{2}|=t$.
\end{lem}

\begin{proof}

Let $H_3'$ be the restriction of $H_3$ to the affine span of $\textbf{0123}$. $H_3'$ is generated by $\textbf{03}$ and the direction vector $e_2-e_4$. Consider the point $\textbf{r}_3$ on the line segment $\textbf{12}$ such that $|\textbf{1}\textbf{r}_3|/|\textbf{r}_3\textbf{2}|=t$, and let $\textbf{s},\textbf{t}$ be the points on the line segments $\textbf{02}$ and $\textbf{03}$ respectively so that $|\textbf{0s}|/|\textbf{s2}|=t$, and $|\textbf{0t}|/|\textbf{t3}|=t$. By direct computation (or similar triangles), we get $\textbf{r}_3-\textbf{t}=(\textbf{r}_3-\textbf{s})+(\textbf{s}-\textbf{t})=\frac{1}{t+1}(e_4-e_2)$. The conclusion follows.

\end{proof}

By a similar argument in $\textbf{1234}$, we get that $\textbf{r}_1=H_1\cap\textbf{23}$ satisfies $|\textbf{2}\textbf{r}_1|/|\textbf{r}_1\textbf{3}|=t$.

Finally, $H_2$ contains the angle bisector $e_2-e_3$ of the right angle $\angle\textbf{123}$, so we get that $\textbf{r}_2=H_2 \cap \textbf{13}$ satisfies $|\textbf{3r}_2|/|\textbf{r}_2\textbf{1}|=|\textbf{32}|/|\textbf{21}|=\frac{1}{t^2}$.

Hence as $\frac{|\textbf{2}\textbf{r}_1|}{|\textbf{r}_1\textbf{3}|}\cdot\frac{|\textbf{3r}_2|}{|\textbf{r}_2\textbf{1}|}\cdot\frac{ |\textbf{1}\textbf{r}_3|}{|\textbf{r}_3\textbf{2}|}=1$, we get by Ceva's theorem (see \cite{Johnson}) that the line segments $\textbf{1r}_1$, $\textbf{2r}_2$, $\textbf{3r}_3$ intersect in a point $\textbf{r}$ (see Figure 9). Hence $H_1\cap H_2 \cap H_3\cap \textbf{0123}$ contains the line segment $\textbf{0r}$. As $\textbf{r}-\textbf{2}$ is a multiple of $e_2-e_3$, we get that the rank of $\textbf{r}$ is the rank of $\textbf{2}$, which is at least the middle rank as $t\ge 1$. Hence there is a point \textbf{B'} on the line segment $\textbf{0r}$ of middle rank.

We can now cone off at $\textbf{B}$ and $\textbf{B'}$, without self-intersection occuring, to get a geometric symmetric chain decomposition with the weak hyperplane condition of the resulting polytope $P_t=\overline{L(4)_t\setminus \text{conv}(\textbf{B'1234})}$ (where conv denotes convex hull). In particular if $t=1$ we get $P_1=L(4)$.

Even though the above geometric argument shows the weak hyperplane condition is satisfied for $P_t$, we can get to the same conclusion without geometric insight using the methods of linear algebra.

Set $\textbf{0}=(0,0,0,0)$, $\textbf{1}=(0,0,0,1)$, $\textbf{2}=(0,0,t,1)$, $\textbf{3}=(0,\frac{1}{t},t,1)$, $\textbf{4}=(1,\frac{1}{t},t,1)$, $\textbf{B}=(\frac{1}{2},\frac{1}{2t},\frac{t}{2},\frac{1}{2})$, $\textbf{B'}=(0,\frac{t+1}{2(t^2+t+1)},\frac{t(t+1)^2}{2(t^2+t+1)},\frac{t+1}{2t})$, and $P_t=\overline{L(4)_t\setminus \text{conv}(\textbf{B'1234})}$.

The table below shows that the weak hyperplane condition is satisfied. We put a $*$ above the directions corresponding to turning sets associated to a single swipe (they do not contribute to the weak hyperplane condition). Hence we get asymptotic decompositions of $P_t(n)$ by Theorem 3.6.

\begin{tabular}{| c | c | c | c | c| c |}
\hline
Turning Set & \# (Real)& \# (Fake)& Complexity & In & Out\\
\hline
$\frac{1}{t}x-y-z+tw=0$ & 0,4,6& 0,2,6&1 & $\text{x}^*$,z & y,$\text{w}^*$ \\
$\frac{t^3+2t^2+1}{t^2(t+1)}x-\frac{t^5+t^4+t^3+t}{t^2(t+1)}y+z+w=1$ &  &  1&1  & w & z\\
$y-\frac{t+1}{t^2}z+w=0$ &  1&  3&1  & y & w\\
$x+\frac{1-t^3}{1+t}y-\frac{1-t^3}{t^2(1+t)}z+w=1$ & 2& 4&1 & w & x\\
$x-t(1+t)y+z=0$ &  3&  5&1  & x & z\\
$x+y-\frac{t^4+t^3+t^2+1}{t^2(t+1)}z+\frac{t^5+2t^3+t^2}{t^2(t+1)}w=1$ & 5 &1 & 1&  y& x\\
\hline
\end{tabular}\\
This concludes the proof of Theorem 3.7. We now finish the subsection with the proof of a result alluded to at the beginning of the subsection.

\begin{thm}
A necessary condition for $L(4)_{a,b,c,d}$ to have an asymptotic geometric symmetric chain decomposition satisfying the weak hyperplane condition is that $a=d$ and $b=c$.
\end{thm}
\begin{proof}
By hypothesis, it follows that the underlying posets of $L(4)_{a,b,c,d}$ have an asymptotic chain decomposition. (Asymptotic) rank-symmetry of the underlying posets then forces the volume $f(\lambda)$ of $L(4)_{a,b,c,d}\cap \{x+y+z+w\le\lambda\}$ to be complementary about the middle rank. As $f(\lambda)$ is piecewise cubic with break points at the 5 vertices, this forces the ranks of the vertices to be complementary.
\end{proof}

\begin{rmk}
The polytope $L(4)_t$ with $t>1$ does not satisfy $a=d$ and $b=c$, and the proof above shows that the obstruction to creating an asymptotic geometric symmetric chain decomposition with the weak hyperplane condition is a failure of ``volume rank-symmetry''. Hence we needed to remove a region to restore this symmetry. It is interesting to note that therefore, the choice of $\textbf{B'}$ happens to be so that the volume function doesn't change when $\lambda$ passes the rank of $\textbf{2}$.
\end{rmk}

\subsection{The polytope $L(5)$}

Set $\textbf{0}=(0,0,0,0,0)$, $\textbf{1}=(0,0,0,0,1)$, $\textbf{2}=(0,0,0,1,1)$, $\textbf{3}=(0,0,1,1,1)$,
$\textbf{4}=(0,1,1,1,1)$,
$\textbf{5}=(1,1,1,1,1)$, and write $\overline{\textbf{ij}}$ for the midpoint of $\textbf{ij}$.

We create a geometric symmetric chain decomposition of $L(5)$ by projecting from the unique points of rank $\frac{5}{2}$ on the line segments $\textbf{13}$, $\textbf{24}$, and $\textbf{05}$ (see Figure 10). The projected polytope is the union of the 8 faces $\textbf{012}$, $\textbf{023}$, $\textbf{034}$, $\textbf{014}$, $\textbf{125}$, $\textbf{235}$, $\textbf{345}$, $\textbf{145}$. This has a geometric symmetric chain decomposition with one real snake $R$, and one fake snake $F$.

The real snake has turning sets $\textbf{0 }\overline{\textbf{02}}$, $\textbf{0 }\overline{\textbf{03}}$, $\textbf{0 }\overline{\textbf{04}}$, $\textbf{1 }\overline{\textbf{14}}$, 
$\textbf{1 }\overline{\textbf{15}}$,
$\textbf{2 }\overline{\textbf{25}}$,
$\textbf{3 }\overline{\textbf{35}}$,
$\textbf{4 }\overline{\textbf{35}}$, $\textbf{5 }\overline{\textbf{35}}$.

The fake snake has turning sets
$\textbf{0 }\overline{\textbf{02}}$, $\textbf{1 }\overline{\textbf{02}}$, $\textbf{2 }\overline{\textbf{02}}$, $\textbf{3 }\overline{\textbf{03}}$, 
$\textbf{4 }\overline{\textbf{04}}$, $\textbf{4 }\overline{\textbf{14}}$,
$\textbf{5 }\overline{\textbf{15}}$,
$\textbf{5 }\overline{\textbf{25}}$,
$\textbf{5 }\overline{\textbf{35}}$.

We note that each of the pairs $\textbf{0}$ and $\textbf{5}$, and $\overline{\textbf{02}}$ and $\overline{\textbf{35}}$ has complementary ranks.

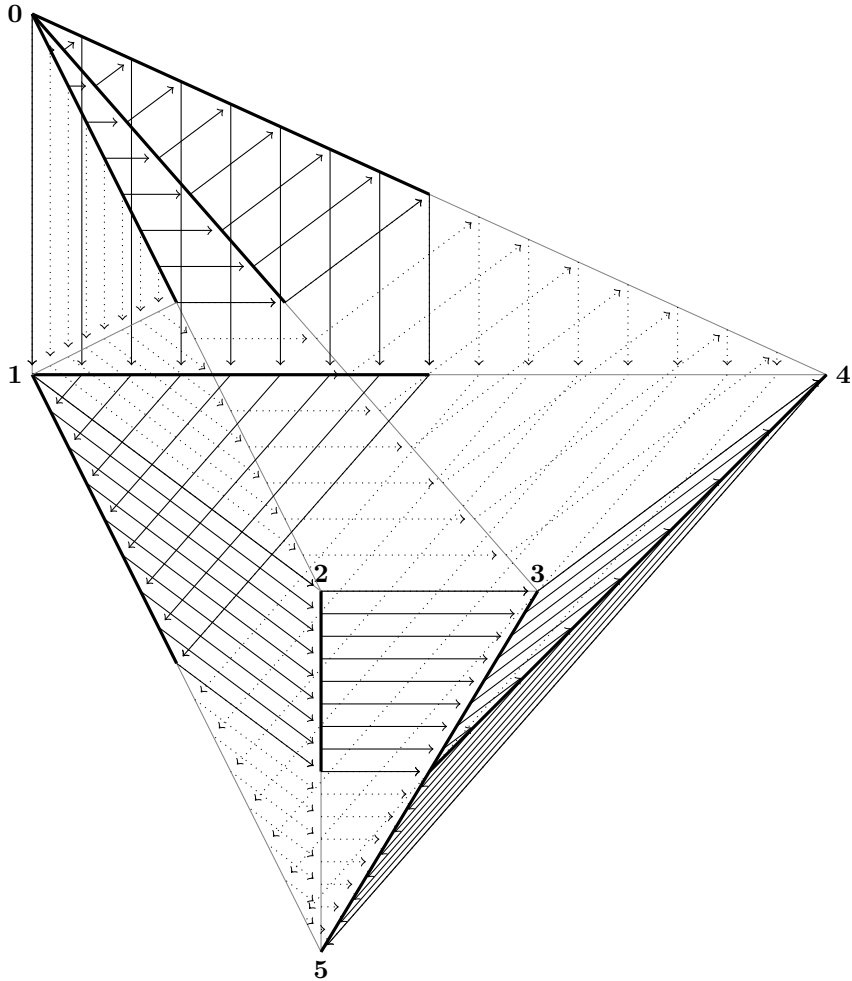
\begin{figure}[H]
\centering
\begin{tikzpicture}[scale=0.96]
\def\ax{0}; \def\ay{8};
\def\bx{0}; \def\by{3};
\def\ccx{4}; \def\ccy{0};
\def\cx{7}; \def\cy{0};
\def\dx{11}; \def\dy{3};
\def\ex{4}; \def\ey{-5};
\draw (\ax,\ay) node[left] {\textbf{0}};
\draw (\bx,\by) node[left] {\textbf{1}};
\draw (\ccx,\ccy) node[above] {\textbf{2}};
\draw (\cx,\cy) node[above] {\textbf{3}};
\draw (\dx,\dy) node[right] {\textbf{4}};
\draw (\ex,\ey) node[below] {\textbf{5}};

\draw[gray, very thin] (\ax,\ay)--(\bx,\by);
\draw[gray, very thin] (\ax,\ay)--(\ccx,\ccy);
\draw[gray, very thin] (\ax,\ay)--(\cx,\cy);
\draw[gray, very thin] (\ax,\ay)--(\dx,\dy);

\draw[gray, very thin] (\ex,\ey)--(\bx,\by);
\draw[gray, very thin] (\ex,\ey)--(\ccx,\ccy);
\draw[gray, very thin] (\ex,\ey)--(\cx,\cy);
\draw[gray, very thin] (\ex,\ey)--(\dx,\dy);

\draw[gray, very thin] (\bx,\by)--(\ccx,\ccy);
\draw[gray, very thin] (\ccx,\ccy)--(\cx,\cy);
\draw[gray, very thin] (\cx,\cy)--(\dx,\dy);
\draw[gray, very thin] (\dx,\dy)--(\bx,\by);

\draw[very thick] (\ax,\ay)--(\ax/2+\ccx/2,\ay/2+\ccy/2);
\draw[very thick] (\ax,\ay)--(\ax/2+\cx/2,\ay/2+\cy/2);
\draw[very thick] (\ax,\ay)--(\ax/2+\dx/2,\ay/2+\dy/2);
\draw[very thick] (\bx,\by)--(\bx/2+\dx/2,\by/2+\dy/2);
\draw[very thick] (\bx,\by)--(\bx/2+\ex/2,\by/2+\ey/2);
\draw[very thick] (\ccx,\ccy)--(\ccx/2+\ex/2,\ccy/2+\ey/2);
\draw[very thick] (\dx,\dy)--(\cx/2+\ex/2,\cy/2+\ey/2);
\draw[very thick] (\cx,\cy)--(\ex,\ey);

\foreach \k in {1,2,...,8}
{
\draw[->, shorten >=0.125cm] ({(\k/8)*(\ccx/2-\ax/2)+\ax},{\ay-(\k/8)*(\ay/2-\ccy/2)})--({(\k/8)*(\cx/2-\ax/2)+\ax},{\ay-(\k/8)*(\ay/2-\cy/2)});
\draw[->, shorten >=0.125cm] ({(\k/8)*(\cx/2-\ax/2)+\ax},{\ay-(\k/8)*(\ay/2-\cy/2)})--({(\k/8)*(\dx/2-\ax/2)+\ax},{\ay-(\k/8)*(\ay/2-\dy/2)});
\draw[->, shorten >=0.125cm] ({(\k/8)*(\dx/2-\ax/2)+\ax},{\ay-(\k/8)*(\ay/2-\dy/2)})--({(\k/8)*(\dx/2-\bx/2)+\bx},{\by-(\k/8)*(\by/2-\dy/2)});
}

{
\def\k{0};
\draw[->, shorten >=0.125cm] ({(\k/8)*(\dx/2-\ax/2)+\ax},{\ay-(\k/8)*(\ay/2-\dy/2)})--({(\k/8)*(\dx/2-\bx/2)+\bx},{\by-(\k/8)*(\by/2-\dy/2)});
}

\foreach \k in {1,2,...,7}
{
\draw[->, dotted, shorten >=0.125cm] ({(\k/8)*(\ccx/2-\ax/2)+\ax},{\ay-(\k/8)*(\ay/2-\ccy/2)})--({(\k/8)*(\ax/2+\ccx/2-\bx)+\bx},{(\k/8)*(\ay/2+\ccy/2-\by)+\by});
\draw[->, dotted, shorten >=0.125cm] ({(\k/8)*(\ax/2+\ccx/2-\bx)+\bx},{(\k/8)*(\ay/2+\ccy/2-\by)+\by})--({(\k/8)*(\ax/2-\ccx/2)+\ccx},{\ccy-(\k/8)*(\ccy/2-\ay/2)});
\draw[->, dotted, shorten >=0.125cm] ({(\k/8)*(\ax/2-\ccx/2)+\ccx},{\ccy+(\k/8)*(\ay/2-\ccy/2)})--({(\k/8)*(\ax/2-\cx/2)+\cx},{\cy+(\k/8)*(\ay/2-\cy/2)});
\draw[->, dotted, shorten >=0.125cm] ({(\k/8)*(\ax/2-\cx/2)+\cx},{\cy+(\k/8)*(\ay/2-\cy/2)})--({(\k/8)*(\ax/2-\dx/2)+\dx},{\dy+(\k/8)*(\ay/2-\dy/2)});
\draw[->, dotted, shorten >=0.125cm] ({(\k/8)*(\ax/2-\dx/2)+\dx},{\dy+(\k/8)*(\ay/2-\dy/2)})--({(\k/8)*(\bx/2-\dx/2)+\dx},{\dy+(\k/8)*(\by/2-\dy/2)});
}

{
\def\k{0};
\draw[->, dotted, shorten >=0.125cm] ({(\k/8)*(\ccx/2-\ax/2)+\ax},{\ay-(\k/8)*(\ay/2-\ccy/2)})--({(\k/8)*(\ax/2+\ccx/2-\bx)+\bx},{(\k/8)*(\ay/2+\ccy/2-\by)+\by});
\draw[->, dotted, shorten >=0.125cm] ({(\k/8)*(\ax/2+\ccx/2-\bx)+\bx},{(\k/8)*(\ay/2+\ccy/2-\by)+\by})--({(\k/8)*(\ax/2-\ccx/2)+\ccx},{\ccy-(\k/8)*(\ccy/2-\ay/2)});
\draw[->, dotted, shorten >=0.125cm] ({(\k/8)*(\ax/2-\ccx/2)+\ccx},{\ccy+(\k/8)*(\ay/2-\ccy/2)})--({(\k/8)*(\ax/2-\cx/2)+\cx},{\cy+(\k/8)*(\ay/2-\cy/2)});
\draw[->, dotted, shorten >=0.125cm] ({(\k/8)*(\ax/2-\cx/2)+\cx},{\cy+(\k/8)*(\ay/2-\cy/2)})--({(\k/8)*(\ax/2-\dx/2)+\dx},{\dy+(\k/8)*(\ay/2-\dy/2)});
}

{
\def\k{8};
\draw[->, dotted, shorten >=0.125cm] ({(\k/8)*(\ax/2-\ccx/2)+\ccx},{\ccy+(\k/8)*(\ay/2-\ccy/2)})--({(\k/8)*(\ax/2-\cx/2)+\cx},{\cy+(\k/8)*(\ay/2-\cy/2)});
\draw[->, dotted, shorten >=0.125cm] ({(\k/8)*(\ax/2-\cx/2)+\cx},{\cy+(\k/8)*(\ay/2-\cy/2)})--({(\k/8)*(\ax/2-\dx/2)+\dx},{\dy+(\k/8)*(\ay/2-\dy/2)});
\draw[->, dotted, shorten >=0.125cm] ({(\k/8)*(\ax/2-\dx/2)+\dx},{\dy+(\k/8)*(\ay/2-\dy/2)})--({(\k/8)*(\bx/2-\dx/2)+\dx},{\dy+(\k/8)*(\by/2-\dy/2)});
}

\foreach \k in {1,2,...,8}
{
\draw[->,dotted, shorten >=0.125cm] ({(\k/8)*(\ccx/2-\ex/2)+\ex},{\ey-(\k/8)*(\ey/2-\ccy/2)})--({(\k/8)*(\cx/2-\ex/2)+\ex},{\ey-(\k/8)*(\ey/2-\cy/2)});
\draw[<-,dotted, shorten <=0.125cm] ({(\k/8)*(\ccx/2-\ex/2)+\ex},{\ey-(\k/8)*(\ey/2-\ccy/2)})--({(\k/8)*(\bx/2-\ex/2)+\ex},{\ey-(\k/8)*(\ey/2-\by/2)});
\draw[<-,dotted, shorten <=0.125cm] ({(\k/8)*(\bx/2-\ex/2)+\ex},{\ey-(\k/8)*(\ey/2-\by/2)})--({(\k/8)*(\bx/2-\dx/2)+\dx},{\dy-(\k/8)*(\dy/2-\by/2)});
}

{
\def\k{0}
\draw[<-,dotted, shorten <=0.125cm] ({(\k/8)*(\bx/2-\ex/2)+\ex},{\ey-(\k/8)*(\ey/2-\by/2)})--({(\k/8)*(\bx/2-\dx/2)+\dx},{\dy-(\k/8)*(\dy/2-\by/2)});
}

\foreach \k in {1,2,...,7}
{
\draw[<-, shorten <=0.125cm] ({(\k/8)*(\cx/2-\ex/2)+\ex},{\ey-(\k/8)*(\ey/2-\cy/2)})--({(\k/8)*(\ex/2+\cx/2-\dx)+\dx},{(\k/8)*(\ey/2+\cy/2-\dy)+\dy});
\draw[<-, shorten <=0.125cm] ({(\k/8)*(\ex/2+\cx/2-\dx)+\dx},{(\k/8)*(\ey/2+\cy/2-\dy)+\dy})--({(\k/8)*(\ex/2-\cx/2)+\cx},{\cy-(\k/8)*(\cy/2-\ey/2)});
\draw[<-, shorten <=0.125cm] ({(\k/8)*(\ex/2-\cx/2)+\cx},{\cy+(\k/8)*(\ey/2-\cy/2)})--({(\k/8)*(\ex/2-\ccx/2)+\ccx},{\ccy+(\k/8)*(\ey/2-\ccy/2)});
\draw[<-, shorten <=0.125cm] ({(\k/8)*(\ex/2-\ccx/2)+\ccx},{\ccy+(\k/8)*(\ey/2-\ccy/2)})--({(\k/8)*(\ex/2-\bx/2)+\bx},{\by+(\k/8)*(\ey/2-\by/2)});
\draw[<-, shorten <=0.125cm] ({(\k/8)*(\ex/2-\bx/2)+\bx},{\by+(\k/8)*(\ey/2-\by/2)})--({(\k/8)*(\dx/2-\bx/2)+\bx},{\by+(\k/8)*(\dy/2-\by/2)});
}

{
\def\k{0};
\draw[<-, shorten <=0.125cm] ({(\k/8)*(\cx/2-\ex/2)+\ex},{\ey-(\k/8)*(\ey/2-\cy/2)})--({(\k/8)*(\ex/2+\cx/2-\dx)+\dx},{(\k/8)*(\ey/2+\cy/2-\dy)+\dy});
\draw[<-, shorten <=0.125cm] ({(\k/8)*(\ex/2+\cx/2-\dx)+\dx},{(\k/8)*(\ey/2+\cy/2-\dy)+\dy})--({(\k/8)*(\ex/2-\cx/2)+\cx},{\cy-(\k/8)*(\cy/2-\ey/2)});
\draw[<-, shorten <=0.125cm] ({(\k/8)*(\ex/2-\cx/2)+\cx},{\cy+(\k/8)*(\ey/2-\cy/2)})--({(\k/8)*(\ex/2-\ccx/2)+\ccx},{\ccy+(\k/8)*(\ey/2-\ccy/2)});
\draw[<-, shorten <=0.125cm] ({(\k/8)*(\ex/2-\ccx/2)+\ccx},{\ccy+(\k/8)*(\ey/2-\ccy/2)})--({(\k/8)*(\ex/2-\bx/2)+\bx},{\by+(\k/8)*(\ey/2-\by/2)});
}

{
\def\k{8};
\draw[<-, shorten <=0.125cm] ({(\k/8)*(\ex/2-\cx/2)+\cx},{\cy+(\k/8)*(\ey/2-\cy/2)})--({(\k/8)*(\ex/2-\ccx/2)+\ccx},{\ccy+(\k/8)*(\ey/2-\ccy/2)});
\draw[<-, shorten <=0.125cm] ({(\k/8)*(\ex/2-\ccx/2)+\ccx},{\ccy+(\k/8)*(\ey/2-\ccy/2)})--({(\k/8)*(\ex/2-\bx/2)+\bx},{\by+(\k/8)*(\ey/2-\by/2)});
\draw[<-, shorten <=0.125cm] ({(\k/8)*(\ex/2-\bx/2)+\bx},{\by+(\k/8)*(\ey/2-\by/2)})--({(\k/8)*(\dx/2-\bx/2)+\bx},{\by+(\k/8)*(\dy/2-\by/2)});
}

\draw[gray, very thin] (\bx,\by)--(\ax/2+\ccx/2,\ay/2+\ccy/2);

\end{tikzpicture}
\caption{Geometric symmetric chain decomposition for the projected polytope of $L(5)$.}
\end{figure}

We give the equations of the hyperplanes in the table below.

\begin{center}
\begin{tabular}{| c | c | c | c | c| c |}
\hline
Turning Set & \# (Real)& \# (Fake)& Complexity & Directions In & Directions Out\\
\hline
$b-c-3d+3e=0$ &0,2,5&0,2,4,7&1&b,d&c,e \\
$4a-3b-c+d+e=1$ &&1&1&e&d \\
$3a-3b-c+d=0$ &1,4,6,8&3,6,8&1&a,c&b,d \\
$a+e=1$ &3&5&1&a&e \\
$a+b-c-3d+4e=1$ &7&&1&b&a \\
\hline
\end{tabular}
\end{center}

From Theorem 3.5, we deduce that there is a constant $M$ (which can be taken to be $27$ by careful inspection of the proof of Theorem 3.5), such that there is a disjoint collection of symmetric chains in $L(5,Mn)$ which pass through all points of $L(5,n)\subset L(5,Mn)$. Notice that neither the weak nor the strong hyperplane condition is satisfied.

This concludes the proof of Theorem 3.1.
\section{Proofs of the Main Results}
\begin{proof}[Proof of Proposition 3.2]
For each real snake, we construct closed chains starting on all points of $S_0^s$, and for each fake snake, we construct open chains similarly. We make the chains follow each of the swipes, turning at each turning set, until we hit the last ending set $S_k^e$. As the snakes cover the polytope, each point is covered by some chain. Any two chains from different snakes are obviously disjoint, and because of the non self-intersecting condition on the swipes within a snake, we get the chains inside a given snake are disjoint. Finally, the chains are symmetric by linear interpolation of the ranks of the vertices of $\bar{S_0^s}$ and $\bar{S_k^e}$ within a given snake.
\end{proof}

\begin{proof}[Proof of Theorem 3.3]
By scaling up the polytope by the least common multiple of all the complexities, we may assume that all complexities are $1$. Consider $P(n)$, and restrict the symmetric chain decomposition of $P$ to this discrete poset. To show that we get a symmetric chain decomposition of the discrete poset, we claim it suffices to show that any (continuous) symmetric chain which intersects $P(n)$ has its starting point, ending point, and all turning points contained in $P(n)$.

Indeed, requiring each such continuous chain to have endpoints in $P(n)$ means that the restriction of the chains to $P(n)$ have the sum of the maximal and minimal ranks correct to be symmetric. Furthermore, if we have that two consecutive turning points (i.e. points which lie on turning sets) of the continuous chain lie in $P(n)$, then on the segment between them, we skip no ranks of points in $P(n)$, which shows the (discrete) chain is symmetric.

Given a point $p$ in $P(n)$ contained in a swipe $S$, we draw a line segment through $p$ in the direction $e_i$ of the swipe, which has endpoints $x^s$ and $x^e$ on the two turning sets of the swipe. We claim that these endpoints are in fact points of $P(n)$. If $x^s=x^e$, then we are done. Otherwise, choose one of these turning sets, say the one through $x^e$ (the other case proceeds similarly). The strong hyperplane condition implies there exists a hyperplane of complexity $1$ containing the turning set which can be written as $H(x)=\sum a_jx_j=b$ with $b$ integral, $a_j$ integers, and $a_i=1$.

Writing $x^e=p+\lambda e_i$, we get $b=H(x^e)=H(p+\lambda e_i)=H(p)+\lambda$. As $b$ and all $a_j$ are integers, and all coordinates of $p$ lie in $\frac{1}{n}\mathbb{Z}$, this shows $\lambda \in \frac{1}{n}\mathbb{Z}$. Hence from the equation $x^e=p+\lambda e_i$, we get $x^e \in P(n)$. The same argument shows $x^s \in P(n)$, and by the discussion above, the conclusion follows.

\end{proof}

\begin{proof}[Proof of Theorem 3.5]
Consider $P(kM)$, where $M$ is an integer to be specified later, and restrict the symmetric chain decomposition of $P$ to this discrete poset. Throw away any chains which don't pass through a point of $P(k)$. We need to show that all chains that remain are symmetric chains. As in the proof of Theorem 3.3, it suffices to show that any geometric chain which intersects $P(k)$ has its starting point, ending point, and all turning points contained in $P(kM)$.

To show this, it is sufficient to show that for any swipe $S$, there exists an integer constant $M_S$ which depends only on $S$ (and not on $k$), such that if $p$ is a point in $P(k)$ contained in the swipe $S$, and we draw a line segment through $p$ in the direction $e_i$ of the swipe, which has endpoints $x^s$ and $x^e$ on the two turning sets of the swipe, then both $x^s$ and $x^e$ lie inside $P(kM_S)$.

Indeed, we can then just take $M$ to be the product of the $M_S$'s over all swipes $S$.

If $x^s=x^e$, then any integral $M_S$ would work for this particular $x$. Otherwise, choose one of these turning sets, say the one through $x^e$ (the other case is identical). As the affine span of the ending turning set is rational, there exists a rational hyperplane containing the affine span and not containing the direction of the swipe $e_i$ which is given by an equation $H(x)=\sum a_jx_j=b$ with all $a_j$ and $b$ integral, and $a_i\ne 0$.

Writing $x^e=p+\lambda e_i$, we get $b=H(x^e)=H(p+\lambda e_i)=H(p)+\lambda a_i$. As all $b$ and all $a_j$ are integers, and all coordinates of $p$ lie in $\frac{1}{k}\mathbb{Z}$, this shows $\lambda \in \frac{1}{ka_i}\mathbb{Z}$. Hence from the equation $x^e=p+\lambda e_i$, we get $x^e \in P(ka_i)$. Let $a_i'$ be the analogous constant for $x^s$, and take $M_S=a_ia_i'$. From the discussion above, the conclusion follows.
\end{proof}

\begin{proof}[Proof of Theorem 3.6]

For each point $x$ lying inside $P(k)$, we construct a chain $C_x$ as follows. If $x$ does not lie inside a snake, then let $C_x=\{ x \}$; otherwise, let $T$ be the snake that $x$ lies inside. Let $S_i$ be the earliest swipe that $x$ lies inside. Start the chain $C_x$ at the point $x$, and then proceed by extending the chain $C_x$ in the direction of the current swipe until it strictly surpasses the ending turning set. If we are not in the next swipe (or if there is no next swipe), then stop the chain right before it exits the current swipe; otherwise, repeat the process of extending the chain.

We claim that the weak hyperplane condition implies that for any two of these chains, either they are disjoint, or one is contained inside the other. To check this, we can restrict ourselves to a single snake and suppose we have chains $C_x$ and $C_y$ associated to the points $x,y$. Suppose $C_x$ and $C_y$ intersect and one is not contained inside the other. Look at the first time they intersect, say at the point $z$ in the swipe $S_r$. Note that $z$ cannot be the first point of either $C_x$ or $C_y$.

Clearly, $z$ is the first point inside $S_r$ of one of the chains, say $C_x$. As $z$ does not start $C_x$, $S_r$ is not the first swipe in the snake. Let $x'$ be the predecessor of $z$ in $C_x$, and $y'$ the predecessor of $z$ in $C_y$ (which both exist). Let $H(x)=\sum a_kx_k=b$ with $a_i=a_j$, where $e_i,e_j$ are the directions of $S_{r-1}$ and $S_r$ respectively, be the equation of the hyperplane containing $S_{r-1}^e=S_r^s$.

Note that $x'=z-\frac{1}{k}e_i$ belongs to $S_{r-1}$. Either $y'$ is in $S_{r-1}$ or $y'$ is in $S_r\setminus S_r^s$. If $y'$ is in $S_{r-1}$, then $y'=z-\frac{1}{k}e_i=x'$, which can't happen by the minimality of $z$. If $y'$ lies in $S_r\setminus S_{r}^s$, then $y'=x'+\frac{1}{k}e_{i}-\frac{1}{k}e_j$, from which we immediately see that $H(x')=H(y')$, which can't happen as $S_{r-1}$ and $S_r$ are on opposite sides of $H$. The claim now follows, i.e. every two such chains are either disjoint, or one is contained inside the other.

In the collection of all chains $C_x$, consider the maximal chains with respect to inclusion of sets. From the discussion above, all such chains are disjoint and partition $P(k)$.

Fix a snake $T$. First, throw away any chains that don't start in the first swipe $S_0$, and don't end in the last swipe $S_l$. From each remaining chain, throw away as few elements as possible from the ends to make the chain a symmetric chain.

We claim that we have thrown away at most $O(\frac{1}{k})$ of the points in $T$. Once we show this, then as there are only finitely many such snakes, and the number of points not in any snake is $O(\frac{1}{k})$ of the number of points in $P(k)$, the theorem follows. To start, we make the following three observations.

\begin{obs}
Take a full dimensional lattice $L$ with basis $\{f_1,f_2,\ldots f_n\}$. Let $H$ be a hyperplane not containing the direction $f_1$. Then if we project the lattice onto $H$ in the direction $f_1$, we get a full dimensional lattice $L'$ inside $H$. Furthermore, if we refine the lattice $L$ by a constant integer factor $C$, then $L'$ is refined by $C$ as well.
\end{obs}

\begin{obs}
By the weak hyperplane condition, we would have attained the same family of chains $C_x$ by traveling backwards instead, and stopping strictly before we hit the turning sets before changing directions.
\end{obs}

\begin{obs}
If we have two hyperplanes $H_1$ and $H_2$ whose non-empty intersection is $H'$ and a direction vector $e$ not parallel to $H_1$, then if $x_2$ is a point on $H_2$, $x_1$ is the projection in direction $e$ of $x_2$ on $H_1$, then $d(x_1,H')\le C\cdot d(x_1,x_2)$ for some constant $C$ independent of $x_2$.
\end{obs}

We want to estimate how many chains $C_x$ stop in some swipe $S_i$, say of direction $e_j$, strictly before the last swipe (similarly, we want to estimate the number of chains which start strictly after the first swipe). Denote by $\alpha_x \in S_i$ the final element of a chain $C_x$ which stops in $S_i$.

Consider the line segment from $\alpha_x$ to $\alpha_x+\frac{1}{k}e_j$. Let $\alpha_1$ be the point where the line segment enters $S_{i+1}$, and $\alpha_2$ the point when it leaves $S_{i+1}$ (these exist as $\alpha_x$ is the last point of $C_x$, and as $S_i$ and $S_{i+1}$ are of full dimension). Then $\alpha_1$ lies on $S_{i+1}^s$, and $\alpha_2$ lies on some facet $\overline{W}$ of $\overline{S_{i+1}}$. For $k$ sufficiently large, we must have the facet $\overline{W}$ intersect the facet $\overline{S_{i+1}^s}$, as if $\overline{W}\cap\overline{S_{i+1}^s}=\varnothing$, then there is a positive distance between these facets. Let the hyperplane containing $S_{i+1}^s$ be denoted $H_1$, and the hyperplane containing $\overline{W}$ be denoted $H_2$. We also note $\overline{W}\cap\overline{S_{i+1}^s}$ is a part of the boundary of $\overline{S_{i+1}^s}$; let $H'$ be the affine span of $\overline{W}\cap\overline{S_{i+1}^s}$. Finally, take $R$ to be the orthogonal projection of $\overline{S_{i+1}^s}$ onto $H'$.

Then by Observation 5.3, taking $x_i=\alpha_i$ and $e=-e_j$ (not parallel to $H_1$), we have $d(\alpha_1,R)=d(\alpha_1,H') \le C\cdot d(\alpha_1,\alpha_2) \le \frac{C}{k}$ for some constant $C$ independent of $k$ and $x$.

By Observation 5.1, if we project $\frac{1}{k}\mathbb{Z}^m$ in the direction $e_j$ onto $H_1$, then we obtain a full-dimensional lattice $\frac{1}{k}\Lambda$ on $H_1$, which $\alpha_1$ belongs to.

\begin{lem}
If we have a full-dimensional lattice $\Lambda$ inside $\mathbb{R}^d$, and a bounded region $R$ of an affine space $H'$ of codimension $\ge 1$, then the number of points of $\frac{1}{k}\Lambda$ within distance $\frac{C}{k}$ of $R$ is at most $O(k^{d-1})$.
\end{lem}

\begin{proof}
We can assume $H'$ is of codimension 1. Pick a basis vector $e$ of $\Lambda$ not parallel to the hyperplane $H'$. Project the lattice $\frac{1}{k}\Lambda$ onto the hyperplane $H'$ in the direction $e$ to obtain by Observation 5.1 a lattice $\frac{1}{k}\Lambda'$ on $H'$. Now, there exists a bounded region $R'$ of $H'$ containing $R$ such that a point $x$ within distance $C$ of $R$ gets projected via $e$ inside $R'$. This $R'$ trivially satisfies that if $x \in \frac{1}{k}\Lambda$, and $d(x,R) \le \frac{C}{k}$, then the projection of $x$ via $e$ lies inside $R'\cap \frac{1}{k}\Lambda'$.

There exists a constant $D>0$ such that for any point $x' \in H'$ and any point $x$ lying on the line $L_{x'}$ through $x'$ with direction vector $e$, we have $d(x,x')=D \cdot d(x,H')$. Take a point $x' \in R' \cap \frac{1}{k} \Lambda$. Then if we take a point $x\in L_{x'}$ with $d(x,R) \le \frac{C}{k}$, then $d(x,x') =D \cdot d(x,H') \le D\cdot d(x,R) \le D\frac{C}{k}$. Hence, the number of points in $\frac{1}{k}\Lambda$ within distance $\frac{C}{k}$ of $R$ is bounded above by $(1+2DC/|e|)\cdot |\frac{1}{k}\Lambda' \cap R'|=O(k^{d-1})$.
\end{proof}

To continue the proof of Theorem 3.6, by Lemma 5.4, applied with ambient space $H_1$, the number of points of $\frac{1}{k}\Lambda$ inside $S_{i+1}^s$ of distance at most $\frac{C}{k}$ from $R$ is $O(k^{m-2})$. These points correspond to chains that we potentially throw away. We obtain such a bound for all pairs of a swipe $S_i$ and a facet $W$ contained in the swipe. Combining these bounds, we see that the number of chains that we throw away from the snake that don't make it to the last swipe is at most $O(k^{m-2})$. Since each chain has at most $O(k)$ points, this implies we throw away at most $O(k^{m-1})$ points.

Now, using Observation 5.2, we can apply a similar argument to show the number of chains which don't start in $S_0$ is at most $O(k^{m-2})$. We remark that for this case, one must use the weak hyperplane condition to show that when a chain (viewed as traveling backwards) turns, it leaves its current swipe. Again we throw away at most an additional $O(k^{m-1})$ points. All remaining chains start in the first swipe and end in the last swipe.

We claim there are only $O(k^{m-1})$ chains which we have not thrown away. Indeed, project $\frac{1}{k}\mathbb{Z}^m$ onto (the affine space containing) $S_0^s$ in the direction of $S_0$ to get a full dimensional lattice $\frac{1}{k}\Lambda_0$ by Observation 5.1. As there are $O(k^{m-1})$ points of this lattice contained in $S_0^s$ itself, and each point corresponds to at most one chain, the claim follows.

Take one of the remaining chains $C_x$, and treat it as a continuous curve parametrized by speed $\frac{1}{k}$. Take the continuous chain $G_x$ through $x$ from the snake containing $x$, and again parametrize it by speed $\frac{1}{k}$. The two chains start at distance at most $\frac{1}{k}$ from each other. Because the two curves are not traveling in the same direction only for a bounded amount of time, they deviate from each other at most $\frac{D}{k}$ for some universal constant $D$ (not depending on $k$ or the chain). Hence we only have to throw away $D+1$ points from one of the ends of the chain to make it symmetric. This shows we throw away $O(k^{m-1})$ points from the remaining set.

Thus the total number of points in the snake we need to throw away so that the remaining chains form a symmetric chain decomposition is at most $O(k^{m-1})$. The conclusion of Theorem 3.6 follows.

\end{proof}

\section{Tools and Techniques}
The following tools should help in the pursuit of polytope decompositions. We investigate afterwards products of polytopes with geometric and asymptotic geometric symmetric chain decompositions. For the remainder of this section, we let $P$ be a polytope of dimension $d$ in $\mathbb{R}^m$ with either
\begin{itemize}
\item a geometric symmetric chain decomposition with the strong hyperplane condition, or
\item an asymptotic geometric symmetric chain decomposition with the weak hyperplane condition (so $d=m$).
\end{itemize}

\begin{defn}
Given a $(d-1)$-dimensional subset $S$ of $\mathbb{R}^m$ and a coordinate direction $e_i$ such that the projection of $S$ onto the hyperplane $x_i=0$ is injective, define the \textit{normalized volume in direction $e_i$} $N_i(S)$ to be the ($(d-1)$-dimensional) Lebesgue Measure of the projection of $S$ onto the hyperplane $x_i=0$. For the starting turning set $T$ of a snake, we denote $T_\lambda$ to be the part of the starting set of rank at most $\lambda$, and $N(T_\lambda)$ to be the \textit{normalized volume} in the direction of the first swipe.
\end{defn}

Note that if a subset $S$ lies inside a hyperplane $\sum a_ix_i=b$ with $a_i=a_j$, then the normalized volume is the same in directions $e_i$ and $e_j$. Also, we remark that shearing a subset in the $e_i$ direction doesn't change its normalized volume. Hence in a given $d$-dimensional snake, by the corresponding hyperplane condition, the normalized volume of the turning sets in the direction of the swipes they are adjacent to are the same. The following observation motivates Definition 6.1.

\begin{obs}
As $n \to \infty$, the number of symmetric chains in $P(n)$ which start below rank $\lambda$ in a given d-dimensional snake is asymptotically $N(T_\lambda)n^{d-1}$.
\end{obs}

From the observation $T_\lambda=T$ when $\lambda$ is the middle rank, Theorem 6.3 immediately follows.

\begin{thm}
The normalized volume of the middle rank slice of $P$ (in any coordinate direction) is equal to $\sum_T N(T)$ over all starting sets $T$ of all snakes.
\end{thm}

Taking the statement ``the symmetric chains (asymptotically) partition the entire poset'' in the limit, we get the volume analogue in Theorem 6.4. The key fact used is that the integral of a linear function over a region is given by the volume of the region times the value at the barycenter. In this case, the linear function is the ``length of the chain''.

\begin{thm}
The volume of a snake with starting set $T$ is equal to $N(T)$ times the length of the chain through the barycenter of the starting set. The volume of $P$ is equal to the sum of this quantity over all snakes.
\end{thm}

Theorem 6.4 is especially easy to use in practice if one notices that the linearity of the ``length of chains'' function on the starting set of a snake, along with the fact that the barycenter of a simplex is the average of its vertices, implies that we can replace the length of the chain through the barycenter with the average of the lengths of the chains starting at each vertex of the starting set. More precise versions of Theorems 6.3 and 6.4 are given in Theorems 6.5 and 6.6, whose proofs we again omit.

\begin{thm}
The normalized volume of the rank $\lambda$ slice of $P$ (in any coordinate direction), with $\lambda$ at most the middle rank, is equal to $\sum_TN(T_\lambda)$ over all starting sets $T$ of all snakes.
\end{thm}

\begin{thm}
The volume of a snake with starting set $T$ below rank $\lambda$ (with $\lambda$ less than middle rank) is equal to $N(T_\lambda)$ times the difference of $\lambda$ and the rank of the barycenter of $T_\lambda$. The volume of $P$ below rank $\lambda$ is equal to the sum of this quantity over all snakes.
\end{thm}

To end the section we prove the following natural result about products of polytopes with asymptotic geometric symmetric chain decompositions (remarking afterwards on the analogous result for geometric symmetric chain decompositions).

\begin{thm}
The product of two polytopes with asymptotic geometric symmetric chain decompositions with the weak hyperplane condition has an asymptotic geometric symmetric chain decomposition with the weak hyperplane condition.
\end{thm}

\begin{proof}
Clearly we can reduce the problem to decomposing the product of two (full-dimensional) fake snakes. Consequently, all chains regarded in this proof are open symmetric chains. The strategy is as follows. Regard each snake as a disjoint family of chains, disregarding any further structure. The product of the two snakes is then a disjoint union of open rectangles, formed by taking the product of a chain in the first snake with a chain in the second snake. We then further decompose each rectangle into a family of open symmetric chains as in Figure 11. Finally, we group together chains which move similarly into snakes, and check the weak hyperplane condition.

Given a chain $C_\textbf{x}$ with starting point $\textbf{x}=\textbf{x}_s$ and ending point $\textbf{x}_e$ in the first snake and a chain $C_\textbf{y}$ with starting point $\textbf{y}=\textbf{y}_s$ and ending point $\textbf{y}_e$ in the second snake, we consider the product $C_\textbf{x} \times C_\textbf{y}$. This can be embedded into the plane as an open rectangle $R(\textbf{x},\textbf{y})$ by mapping the point $(\textbf{x'},\textbf{y'}) \in C_\textbf{x} \times C_\textbf{y}$ to the point $(rk(\textbf{x'}),rk(\textbf{y'}))$. We decompose $C_\textbf{x} \times C_\textbf{y}$ into open symmetric chains by decomposing $R(\textbf{x},\textbf{y})$ as in Figure 11, making all the chains first go right, and then up. In the rectangle $R(\textbf{x},\textbf{y})$, consider the angle bisector $l(\textbf{x},\textbf{y}) \subseteq R(\textbf{x},\textbf{y})$ of the bottom right vertex $(rk(\textbf{x}_e), rk(\textbf{y}_s))$ of $R$. This corresponds in $C_\textbf{x} \times C_\textbf{y}$ to the set of points of the same rank as $(\textbf{x}_e, \textbf{y}_s)$. In $R(\textbf{x},\textbf{y})$, we consider the vertical line $v_i$ whose $x$-coordinate is the rank of the point where $C_\textbf{x}$ intersects the $i$'th turning set of the first snake. Similarly, we consider the horizontal line $h_j$ whose $y$-coordinate is the rank of the point where $C_\textbf{y}$ intersects the $j$'th turning set of the second snake. Define $R_{ij}(\textbf{x},\textbf{y})$ to be the region of $R(\textbf{x},\textbf{y})$ bounded by the lines $v_i,v_{i+1},h_j,h_{j+1}$. Notice that in $C_\textbf{x} \times C_\textbf{y}$, the preimage of $R_{ij}(\textbf{x},\textbf{y})$ is a (possibly degenerate) 2-dimensional planar rectangle identical to $R_{ij}(\textbf{x},\textbf{y})$ in the directions of the corresponding swipes.

\begin{figure}[H]
\centering
\begin{tikzpicture}
\draw[gray, very thin] (0,0)--(6,0)--(6,4)--(0,4)--cycle;
\draw[very thin] (1,0)-- node[below] {Chain 1} (5,0);
\draw[very thick] (2,4)--(6,0);
\draw[very thick] (0,0)-- node[left] {Chain 2} (0,4);
\draw[very thick] (2,4)--(6,4);
\foreach \x in {0,1,2,3,4,5,6,7,8}
{
\draw [->, shorten >=0.125cm] (0,\x/2)--(6-\x/2,\x/2);
}
\foreach \x in {1,2,3,4,5,6,7,8}
{
\draw [->, shorten >=0.125cm] (2+\x/2,4-\x/2)--(2+\x/2,4);
}
\foreach \x in {0,6,10,12,26,34}
{
\draw [dashed, very thin] (\x/7,0)--(\x/7,4);
}
\foreach \x in {0,3,6,12,20,24}
{
\draw [dashed, very thin] (0,\x/7)--(6,\x/7);
}
\draw   (0,0) node[left] {$h_0$};
\draw   (0,3/7) node[left] {$h_1,h_2,h_3$};
\draw   (0,6/7) node[left] {$h_4$};

\draw   (0,0) node[below] {$v_0$};
\draw   (6/7,0) node[below] {$v_1$};
\draw   (10/7,0) node[below] {$v_2$};
\end{tikzpicture}
\caption{Product of two chains, decomposed in the same way as in Figure 3.}
\end{figure}
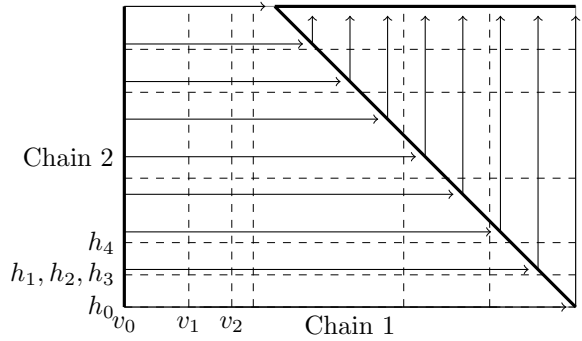

The open symmetric chains in $C_\textbf{x} \times C_\textbf{y}$, when viewed in $R(\textbf{x},\textbf{y})$, make a right angled turn precisely at the intersections with the lines $v_i,h_j$, and $l(\textbf{x},\textbf{y})$. Note that a starting point of a chain in the product space is of the form $(\textbf{x},\textbf{y'})$, where $\textbf{x}$ is a starting point of a chain $C_\textbf{x}$ in the first snake, and $\textbf{y'}$ is a point on a chain $C_\textbf{y}$ in the second snake. This in particular implies that the set of all starting points lies on a hyperplane. Consider the set of starting points $(\textbf{x},\textbf{y'})$ with $\textbf{y'}$ lying in the $i$'th swipe of the second snake, i.e. in $R(\textbf{x},\textbf{y})$, $(\textbf{x},\textbf{y'})$ lies between $h_i$ and $h_{i+1}$. Then as a function of the starting point $(\textbf{x},\textbf{y'})$, the ranks of the start and end points of $C_\textbf{x}$ and $C_\textbf{y}$ are (affine)-linear, as is the rank of $(\textbf{x},\textbf{y'})$ (the coordinates of $v_i$ and $h_j$ are also (affine)-linear, but we don't need to use this). Indeed, $\textbf{x}$ and $\textbf{y'}$ depend linearly on $(\textbf{x},\textbf{y'})$, and $\textbf{y}$ depends (affine)-linearly on $\textbf{y'}$ when $\textbf{y'}$ is restricted to the $i$'th swipe of the second snake. In particular, the distance a chain travels until it hits $l(\textbf{x},\textbf{y})$ is (affine)-linear in the starting position, as this equals $rk(\textbf{x}_e)+rk(\textbf{y}_s)-rk(\textbf{x})-rk(\textbf{y'})$. Group together the chains which cross the same $v_i$'s and $h_j$'s (the chains which are flush with a segment of $v_i$ or $h_j$ form a codimension 1 set). Notice that this in particular forces all of the chains to lie for some fixed $s$ between $h_s$ and $h_{s+1}$. By construction, for some $r,s$, the chains will intersect in order the lines $v_0,v_1,\ldots, v_r, l(\textbf{x},\textbf{y}), h_{s+1}, h_{s+2}, \ldots, h_k$. The set of starting points of each such group of chains is cut out by linear inequalities, so is a polytope. These polytopes cover the starting set, so there exists a finer partition of the starting set into partial simplices.

Consider the set of chains starting at one such partial simplex. We claim this is a snake satisfying the weak hyperplane condition. Indeed, we will show that the collection of turning points of these chains at $v_i$, at $l(\textbf{x},\textbf{y})$, and at $h_j$, are contained in hyperplanes satisfying the weak hyperplane condition.

First, for a fixed $i$, we note that the $v_i$ are contained in the hyperplane $V_{i}$ containing the product of the $i$'th turning set of the first snake with the ambient space of the second snake. This hyperplane clearly satisfies the weak hyperplane condition since the $i$'th turning set of the first snake satisfied the condition. Similar considerations apply for $h_j$. All that remains now is to consider the set of turning points on $l(\textbf{x},\textbf{y})$.

For the set of chains we're considering, $v_r$ is the last vertical segment the chains hit before hitting $l$. By what we described above, the starting simplex is sheared from one $V_i$ to the next in the directions of the corresponding swipes until it hits $V_{r}$, mapping the starting simplex linearly onto $V_r$. If we can show that the distance the chains travel in the direction of the corresponding swipe from $v_r$ to $l$ is a linear function of the starting position in the starting simplex, this will show that the turning points on $l$ are contained in a hyperplane. But this distance is the difference of the distance the chain travels to $l$ (which we've shown to be linear), and the distance the chain travels to $v_{r}$ (which is linear as the swipes transport the starting simplex on $V_0$ linearly to $V_{r}$). Hence, we get the set of turning points on $l(\textbf{x},\textbf{y})$ is contained in a hyperplane. This hyperplane satisfies the weak hyperplane condition because the 2-dimensional slices in the going in and going out directions are given by $R_{rs}(\textbf{x},\textbf{y})$, where $h_{s+1}$ is the next place the chains turn, and the turning set is $l$, which has slope $-1$ so contains the required direction vector.
\end{proof}

The above proof can be applied exactly as written to show the following theorem.

\begin{thm}
Suppose we have two polytopes $P,Q$ with geometric symmetric chain decompositions with the strong hyperplane condition. Suppose further that all snakes in the decompositions of $P$ and $Q$ are real, or all are fake. Then the product $P \times Q$ has a geometric symmetric chain decomposition with the strong hyperplane condition.
\end{thm}

\begin{cor}
An axis-parallel cuboid with rational vertices has a geometric symmetric chain decomposition with the strong hyperplane condition.
\end{cor}

One might think that given two polytopes $P,Q$ with geometric symmetric chain decompositions satisfying the strong hyperplane condition, the product $P\times Q$ has such a decomposition, but this is not always the case. This is surprising, as given a symmetric chain decomposition of $P(n)$ and of $Q(n)$, we get an induced decomposition of $(P \times Q)(n)$ given by decomposing the products of the symmetric chains in a ``consistent way'' as in Theorem 6.7. However, in certain cases there is an obstruction to even creating a symmetric chain decomposition of $P\times Q$, which happens when we try to decompose the product of an open chain with a closed chain. Indeed, we see this phenomenon happen in Example 6.10.

\begin{exmp}
Let $P$ be the boundary of the unit square in $\mathbb{R}^2$ with vertices $(0,0)$, $(0,1)$, $(1,1)$, $(1,0)$, and $Q$ be the unit line segment $[0,1] \subseteq \mathbb{R}$. Then $P$ and $Q$ have geometric symmetric chain decompositions satisfying the strong hyperplane condition, but their product does not even have a symmetric chain decomposition.

\begin{proof}
Clearly the boundary of a unit square has a geometric symmetric chain decomposition with the strong hyperplane condition given by one closed chain and one open chain, and the unit line segment has one trivially as well. The product polytope $P \times Q$ is given by the boundary of the unit cube in $\mathbb{R}^3$ with the interior of the top and bottom faces omitted. Assume $P\times Q$ has a symmetric chain decomposition. Consider the vertex $(0,0,1)$. It has rank $1$, so there must be a chain which passes through it that attains rank at least 2. The part of the chain which is past $(0,0,1)$ either follows one side of the top boundary square, or the other. Suppose it contains the line segment from $(0,0,1)$ to $(0,1,1)$. A chain which contains a point of the form $(\epsilon,0,1)$ for $\epsilon$ very small must contain the line segment from $(\epsilon,0,1)$ to $(1-\epsilon,0,1)$. As the point $(0,0,1)$ was already used in a chain, the only way for this to happen for all arbitrarily small $\epsilon$ is for there to be an open chain from $(0,0,1)$ to $(1,0,1)$. But then the point $(1,0,1)$ must be used in a different chain, and following it backwards, it must zig-zag inside the rectangle with vertices $(0,0,0), (1,0,0), (0,0,1),(1,0,1)$ until it hits rank 1 at the point $(1-t,0,t)$ for some $0\le t<1$. But then any point $(1-s,0,s)$ with $t<s< 1$ cannot be extended to rank $2$ without intersecting an existing chain, contradiction!
\end{proof}
\end{exmp}

\section{Concluding Remarks}
As we are working with lattice points inside polytopes, the theory of Ehrhart polynomials is obviously applicable to our situations --- in particular, such tools could be of great use in finding geometric symmetric chain decompositions with the strong hyperplane condition, e.g. by looking at numbers of lattice points on rank hyperplane slices in more detail than was done in Section 6 (which only took into count the highest order information).

Also, given the obstruction to decomposing products of open chains with closed chains, it would be interesting to know if products such as $L(3) \times L(3)$, $L(3)\times L(4)$, or $L(4) \times L(4)$ have geometric symmetric chain decompositions (or even symmetric chain decompositions), as the only known decompositions of each factor contain both open and closed chains. It would be interesting as well to see our framework expanded to include ``zigzag'' moves, repeatedly alternating coordinate directions to accomplish a diagonal movement.

Finally, it would be interesting to extend the methods used in this paper to prove results about possibly weighted versions of $L(m)$ for $m \ge 5$. As the dimension increases, the technical obstruction becomes not only ensuring the hyperplane conditions are satisfied, but also that no self-intersection occurs when we cone off low dimensional shells. In fact, for $m \ge 6$, we can find at most $\frac{m+1}{2}$ points inside $L(m)$ that we can project from due to the geometrical constraint that the smallest faces two projection points are contained in have no common vertices. Because of this, we do not know if there is a (geometric) symmetric chain decomposition of the polytope $L(m)$ for $m \ge 6$. An alternative direction to explore would be to try and find points to cone off particularly nice low dimensional shells (like the analogue of the two dimensional shells that appear in this paper in higher dimensions) to produce interesting families of polytopes with various types of decompositions.

\end{document}